\documentclass[11pt]{article}
\pdfoutput=1
\usepackage[letterpaper]{geometry}

\usepackage[utf8]{inputenc}
\usepackage[english]{babel}
\usepackage{graphicx}
\usepackage{hyperref}
\usepackage{amsmath}
\usepackage{mathtools}
\usepackage{amsfonts}
\usepackage{amsthm}
\usepackage{parskip}
\usepackage{bbm}
\usepackage{todonotes}
\providecommand{\keywords}[1]{{\textit{Key words.}} #1}
\providecommand{\amsclass}[1]{{\textit{AMS subject classifications.}} #1}

\everymath{\displaystyle}

\makeatletter
\def\thm@space@setup{%
  \thm@preskip=\parskip \thm@postskip=0pt
}
\makeatother

\title{Dynamics of delayed neural field models in two-dimensional spatial domains}

\newcounter{savecntr}
\newcounter{restorecntr}

\author{L. Spek\setcounter{savecntr}{\value{footnote}}\thanks{Department of Applied Mathematics, University of Twente, P.O. Box 217, 7500 AE, Enschede, The Netherlands (\href{mailto:l.spek@utwente.nl}{l.spek@utwente.nl, koen.dijkstra@gmx.net, s.a.vangils@utwente.nl}).}, \,    
  M.~Polner\thanks{Corresponding author. Bolyai Institute, University of Szeged, H-6720 Szeged, Aradi v\'ertan\'uk tere 1, Hungary (\href{mailto:polner@math.u-szeged.hu}{polner@math.u-szeged.hu}).}, \, 
  K. Dijkstra\setcounter{restorecntr}{\value{footnote}}%
  \setcounter{footnote}{\value{savecntr}}\footnotemark
  \setcounter{footnote}{\value{restorecntr}},  \,
  S.~A.~van~Gils\setcounter{restorecntr}{\value{footnote}}%
  \setcounter{footnote}{\value{savecntr}}\footnotemark
  \setcounter{footnote}{\value{restorecntr}} 
  }

\newtheorem{thm}{Theorem}[section]
\newtheorem{cor}[thm]{Corollary}
\newtheorem{lem}[thm]{Lemma}
\newtheorem{prop}[thm]{Proposition}
\newtheorem{defin}[thm]{Definition}

\DeclareMathOperator\rank{rank}

\DeclareMathOperator{\imag}{Im}
\DeclareMathOperator{\real}{Re}
\DeclareMathOperator{\nr}{nr}

\newcommand{\R}{\mathbb R}
\newcommand{\ppf}[2]{\frac{\partial #1}{\partial #2}}

\date{}

\begin{document}

\maketitle

\begin{abstract}
  Delayed neural field models can be viewed as a dynamical system in an appropriate functional analytic setting. On two dimensional rectangular space domains, and for a special class of connectivity and delay functions, we describe the spectral properties of the linearized equation. We transform the characteristic integral equation for the delay differential equation (DDE) into a linear partial differential equation (PDE) with boundary conditions. We demonstrate that finding eigenvalues and eigenvectors of the DDE is equivalent with obtaining nontrivial solutions of this boundary value problem (BVP). When the connectivity kernel consists of a single exponential, we construct a basis of the solutions of this BVP that forms a complete set in $L^2$. This gives a complete characterization of the spectrum and is used to construct a solution to the resolvent problem. As an application we give an example of a Hopf bifurcation and compute the first Lyapunov coefficient.\\
  \\
  \\
  \\
  \amsclass{34K06, 34K08, 92C20, 34B24, 42C15, 42C30, 46A35, 46B15}
\end{abstract}

\section{Introduction}
Neural field models are based on the seminal work of Wilson and Cowan \cite{wilson_excitatory_1972,wilson_mathematical_1973} on the dynamical properties of two populations of excitatory and inhibitory neurons. Instead of taking individual spiking neurons, a neural field model is obtained by spatial and temporal averaging of the membrane potential across a population of neurons over a time interval. The interactions between neurons across synapses are modelled as a convolution over a so-called connectivity kernel and with a nonlinear activation function. In the work of Amari \cite{amari_dynamics_1977}, this model is consolidated into a single integro-differential equation. Later Nunez \cite{nunez_brain_1974} expanded this work by including the transmission delays of the signals between neurons. These neural fields prove to be useful to understand various neural activity in the cortex and other parts of the brain \cite{coombes_stephen_delays_2009,coombes_large-scale_2010,coombes_2014,venkov_dynamic_2007}. 

Delayed neural field models take the form of an integro-differential equation with space dependent delays. By choosing the proper state space, they can be reformulated as an abstract delay differential equation \cite{vg,spek2019neural}, where many available functional analytic tools can be applied. When the neuronal populations are distributed over a one-dimensional domain and a special class of connectivity functions is considered, a quite complete description of the spectrum and resolvent problem of the linearized equation is known \cite{koen, vg}. Recently, this model has been extended by including a diffusion term into the neural field, which models direct, electrical connections, \cite{spek2019neural}. We analyze the evolution of a delayed neural field equation corresponding to a single population of neurons on a two-dimensional spatial domain. For a summary of some extensions of neuronal activity models from one to two dimensions cf. \cite{coombes_2014} and references therein. Numerical methods developed for the efficient and accurate time simulation of neural fields on higher dimensional domains can be found in \cite{faugeras,evelyn,PM}. Moreover, numerical studies of the non-essential spectrum of abstract delay differential equations are also available, \cite{VERMIGLIO_2016}. Analytic results in this framework, to the best of our knowledge, cannot be found in the literature. In \cite{visser_2017}, Visser et al. have characterized the spectrum for a neural field with transmission delays on a spherical domain and computed normal form coefficients of Hopf and double Hopf bifurcations. In this paper we give an analytic description of the spectrum of the linearized problem on a rectangle. Due to our choice of the connectivity kernel and the transmission delays, it is possible to transform the characteristic equation of the DDE into a linear PDE. 

The first step towards finding a solution is to determine the characteristic polynomial for the PDE. We define an equivalence class on the complex plane characterized by the roots of this polynomial. It is then possible to give a partition of the exponential solutions of the PDE corresponding to these equivalence classes. Moreover, when we consider finite linear combinations of exponential solutions of the PDE, we can derive further conditions, identified as boundary conditions. When the connectivity kernel consists of a single exponential, modeling a population with inhibitory neurons, we give a complete characterization of the solution of this boundary value problem (BVP), hence obtain the eigenfunctions corresponding to the eigenvalues of the DDE. In this special connectivity case, the solution of the BVP can be given using the separation of variables, which leads to two Sturm-Liouville problems on one dimensional domains. The vector space of separable solutions form a complete basis in the space of square integrable functions on the rectangle. Using this unique basis expansion, we give a complete characterization of the spectrum and resolvent and this is one of the main results of this paper.   


The paper is organized as follows. In Section~\ref{sec:Func_anal}, we summarize the functional analytic setting from \cite{vg,spek2019neural} that casts the integro-differential equation into an abstract DDE. It is shown in Section~\ref{section spectrum} how to obtain an explicit representation of some eigenvectors of the linearized problem for a particular choice of connectivity function, expressed as a finite linear combination of exponential functions. This type of connectivity models a mixed population of interacting excitatory and inhibitory neurons. In Section~\ref{sec:single_exp} we give a complete description of the spectrum and resolvent problem when the connectivity is a single exponential. Finally, we show  an example of a Hopf bifurcation in Section~\ref{s:Hopf_bifurcation}. 

The mathematical model for neural fields with space-dependent delays is as follows. 
Consider $p$ populations consisting of neurons distributed over a bounded, connected domain $\Omega\subset \mathbb{R}^d,$ $d=1,2,3.$ For each $i,$ the variable $V_i(t,r)$ denotes the membrane potential at time $t,$ averaged over those neurons in the $i$th population 
positioned at $r\in\Omega.$ These potentials are assumed to evolve according to the following system of integro-differential equations 
\begin{equation}\label{NF}
\frac{\partial V_i}{\partial t}(t,r)=-\alpha_i V_i(t,r)+\sum_{j=1}^{p}\int_\Omega J_{ij}(r,r',t)S_j(V_j(t-\tau_{ij}(r,r'),r'))d\, r',
\end{equation}
for $i=1,\dots, p.$ The intrinsic dynamics exhibits exponential decay to the baseline level $0,$ as $\alpha_i>0.$ The propagation delays $\tau_{ij}(r,r')$ measure the time it takes for a signal sent by a 
type-$j$ neuron located at position $r'$ to reach a type-$i$ neuron located at position $r.$ 
The function $J_{ij}(r,r',t)$ represents the connection strength between population $j$ at location $r'$ and population $i$ at location $r$ at time $t.$ 
The firing rate functions are $S_j.$ For the definition and interpretation of these functions we refer to \cite{veltz}. 

\section{Functional analytic setting}
\label{sec:Func_anal}

In this paper we analyze the evolution of a single population of neurons, $p=1,$ in a bounded two-dimensional domain $\Omega\subset \R^2,$ 
\begin{equation}\label{NF-1pop_model}
  \frac{\partial V}{\partial t}(t,r)=-\alpha V(t,r)+\int_\Omega J(r,r')S(V(t-\tau(r,r'),r'))d\, r',\quad \alpha>0.
\end{equation}
Note that we will only deal with autonomous systems, hence, the connectivity does not depend on time. We assume that the following hypotheses are satisfied for the functions involved in the system, (as in~\cite{vg}): the connectivity kernel $J\in C(\bar\Omega\times\bar\Omega),$ 
the firing rate function $S\in C^\infty (\R)$ and its $k$th Fr\'echet derivative is bounded for every $k\in \mathbb{N}_0$ and the delay function $\tau\in  C(\bar\Omega\times\bar\Omega)$ is non-negative.

From the assumption on the delay function $\tau,$ we may set
\[
0< \tau_{max}=\sup_{(r,r')\in\bar\Omega\times\bar\Omega} \tau(r,r')<\infty.
\]
We define the Banach spaces $Y:=C(\bar\Omega, \mathbb{R})$ and  $X:=C\left([-\tau_{max},0];Y\right).$ For $\varphi\in X,$ $s\in[-\tau_{max},0]$ and for $r\in\Omega$ we write $\varphi(s)(r)=\varphi(s,r),$ and its norm is given by
\[
\| \varphi\|_X =\sup_{s\in[-\tau_{max},0]}\| \varphi(s,\cdot)\|_Y,
\]
where $\| \varphi(s,\cdot)\|_Y=\sup_{r\in\Omega}|\varphi(s,r)|.$ From the assumption on the connectivity kernel, it follows that it is bounded in the following norm
\[
\Vert J\Vert_C = \sup_{(r,r')\in\bar\Omega\times\bar\Omega}|J(r,r')|.
\]
We use the traditional notation for the state of the system at time $t$ 
\[
V_t(s)=V(t+s)\in C(\bar\Omega),\quad s\in[-\tau_{max},0],\ t\geq 0.
\]
Define the nonlinear operator $G:X\to Y$ by
\begin{equation}
G(\varphi)(r):=\int_\Omega J(r,r')S\left(\varphi(-\tau(r,r'),r')\right)d\, r'.
\end{equation}
Then the neural field equation (\ref{NF-1pop_model}) can be written as a DDE as
\begin{equation}\label{DDE_G}
\ppf{V}{t}(t)=-\alpha V(t)+G(V_t),
\end{equation}
where the solution is an element of $C([-\tau_{max},\infty);Y)\cap C^1([0,\infty);Y).$ Similarly, we have the 
state of the solution at time $t$ defined as $V_t(s)(x)=V(t+s,x),$ 
$s\in[-\tau_{max},0],$ $t\geq 0,$ $x\in\bar\Omega.$ It was shown in \cite{vg} that under 
the above assumptions on the connectivity, the firing rate function and delay, the operator $G$ is well-defined and it satisfies a global Lipschitz condition.  

Let $DG(\hat{\varphi})\in\mathcal{L}(X,Y)$ be the Fr\'echet derivative of $G$ at the steady state $\hat\varphi\in X$, given as 
\[
DG(\hat\varphi)(\varphi)(r)=\int_{\Omega}J(r,r')S'(\hat\varphi(-\tau(r,r'),r'))\varphi(-\tau(r,r'),r')dr'.
\]
We assume that $S(0)=0,$ such that \eqref{NF-1pop_model} admits the trivial equilibrium. Then the linearized problem around the $\hat\varphi\equiv 0$ equilibrium is
\begin{equation}
    \begin{cases}\label{NF-linear}
    \dot V (t)=-\alpha V(t)+DG(0)V_t, & t\geq 0\\
    V(t)=\varphi(t), & t\in[-\tau_{max},0].
\end{cases}   
\end{equation}
The solution of the linear problem defines a strongly continuous semigroup $T$ on $X$ generated by $A: D(A)\subset X\to X$, where
\[
D(A)=\{\varphi\in X: \varphi'\in X \text{ and } \varphi'(0)=-\alpha\varphi(0)+DG(0)\varphi\},\quad A\varphi=\varphi'.
\]
We denote by $\rho(A),$ $\sigma(A)$ and $\sigma_p(A)$ the resolvent set, the spectrum and 
the point spectrum of $A,$ respectively. 
\section{Spectral properties of the linearized equation}\label{section spectrum}
In this section, we study the spectral properties of the linearized equation \eqref{NF-linear} when the space domain is the rectangle $\bar\Omega=[-a,a]\times [-b,b]$ and the connectivity kernel is a finite linear combination of exponentials of the form
\begin{equation}\label{connectivity}
J(r,r'):=\sum_{i=1}^N \hat{c}_i e^{-\xi_i\|r-r'\|_1}\quad \forall r,r'\in\bar\Omega,
\end{equation}
where $\hat{c}_i, \xi_i\in\mathbb{C}$, such that $J$ is real-valued. Moreover, the delay function is
\begin{equation}\label{delay function}
\tau(r,r'):=\tau_0+\|r-r'\|_1,\  \forall r,r'\in\bar\Omega,\ \tau_0>0.
\end{equation}

First we deal with the essential spectrum, $\sigma_{ess}(A)$, the part of the spectrum which is invariant under compact perturbations. We can leverage the fact that $DG(0)$ is compact with Theorem 27 of \cite{spek2019neural} to find $\sigma_{ess}(A)=\{-\alpha\}$.

The remaining point spectrum $\sigma_p(A)= \sigma(A)\setminus \sigma_{ess}(A)$ consists of eigenvalues with a finite-dimensional eigenspace. Due to Proposition VI.6.7 of \cite{engel_one-parameter_1999}, eigenvectors $\varphi\in X$ for delay equations have the form
\begin{equation}
\varphi(t)(r)=e^{z t}q(r),
\end{equation} 
with the eigenvalue $z\in\mathbb{C}$ and $q\in Y$ a non-trivial solution of the characteristic equation
\begin{equation}\label{char_eq}
\Delta(z)q:= (z+\alpha)q - \sum_{i=1}^N K_i(z)q =0,
\end{equation}
with the linear operators $K_i(z):Y\to Y$ given by
\begin{equation}\label{integral_operator}
\left(K_i(z)q\right)(r) := c_i(z) \int_\Omega e^{-k_i(z)\|r-r'\|_1}q(r') dr', \quad i=1,2,\dots, N,
\end{equation}
with $k_i(z)=z+\xi_i$ and $c_i(z)=\hat{c}_iS'(0)e^{-\tau_0 z}\neq 0$. The integral equation $\Delta(z)q=0$ is a Fredholm integral equation of the second type acting on multivariate functions.

\subsection{From an integral equation to a partial differential equation}\label{s:BVP} The main idea to solving $\Delta(z)q=0$ is to transform the integral equation to a PDE. Some exponential solutions $q$ of the PDE, with some conditions on the boundary, are then again solutions of the characteristic equation \eqref{char_eq}.

First we establish that the eigenvectors are smooth. 
\begin{prop}\label{prop:smooth}
For any $z\in\mathbb{C}\setminus\{-\alpha\},$ the solution $q\in Y$ of $\Delta(z)q=0$ is $q\in C^\infty(\bar\Omega).$
\end{prop}
\begin{proof}
The range of $K_i(z)$ is contained in $C^2(\bar\Omega)$ for all $i=1,\dots,N$ and all $z\in\mathbb C$. Hence, any solution of $\Delta(z)q=0$ is in $C^2(\bar\Omega)$. The result follows by induction. 
\end{proof}
For the remaining part of this section we assume that $q\in C^{\infty}(\bar\Omega)$, so all differential operators applied to $q$ are well-defined.

Differentiating the kernel functions in the integral equation \eqref{char_eq} in the distributional sense w.r.t. one of the spatial variables yields 
\begin{equation}\label{partial_deriv}
\frac{\partial^2}{\partial x}e^{-k_i(z)\|r-r'\|_1}=\left[ k_i^2(z)-2k_i(z)\delta(x-x')\right]e^{-k_i(z)\|r-r'\|_1},
\quad j=1,2,\ i=1,\dots, N,
\end{equation}
with $r=(x,y)$. This motivates the introduction of the differential operators 
\begin{equation}
L_i(z) = \left(k_i^2(z) - \frac{\partial^2}{\partial x} \right)\circ \left(k_i^2(z) -\frac{\partial^2}{\partial y} \right),\ i=1,\dots, N.
\end{equation}
When applying $L_i(z)$ to the integral operator $K_i(z)$ defined in \eqref{integral_operator}, we obtain 
\begin{equation}\label{key_property} 
L_i(z)K_i(z)q=4 c_i(z) k_i^2(z)q \quad \forall q\in Y, \ i=1,\dots, N. 
\end{equation}
So $L_i(z)$ acts like a left-inverse of $K_i(z)$. Using this key property, we find that applying the operator $L(z)=\prod_{i=1}^N L_i(z)$ to the characteristic equation \eqref{char_eq}, leads a the linear constant coefficient PDE 
\begin{equation}\label{PDE} 
L(z)\Delta(z)q = (z+\alpha)\prod_{i=1}^N L_i(z)q
- 4 \sum_{i=1}^N c_i(z)\, k_i^2(z)\prod_{\substack{j=1,\\ j\not=i}}^N L_j(z) q=0.
\end{equation}

We can only generically expect to find solutions of this PDE which are exponential functions or linear combinations of exponential solutions. So we first look for a solution of this PDE in the form
\begin{equation}\label{q_char} 
q(x,y)=e^{\rho x}e^{\nu y},\quad \rho,\nu\in\mathbb{C}, (x,y)\in\Omega.
\end{equation}
This leads to the characteristic polynomial equation $P_z(\rho,\nu)=0$, with $P_z: \mathbb{C}^2\to \mathbb{C}$ given by
\begin{equation}\label{char_pol}
P_z(\rho, \nu)=(z+\alpha)\prod_{i=1}^N(k_i^2(z)-\rho^2)(k_i^2(z)-\nu^2)
-4\sum_{i=1}^N c_i(z)\, k_i^2(z)\prod_{\substack{j=1,\\ j\not=i}}^N (k_j^2(z)-\rho^2)(k_j^2(z)-\nu^2).
\end{equation}
The characteristic polynomial is symmetric under interchanging and negating $\rho$ and $\nu$, i.e., $P_z(\rho,\nu)=P_z(\nu,\rho)=P_z(-\rho,\nu)$. 

The next proposition shows that there are only a few $z\in\mathbb{C}$ such that $P_z(\rho,\nu)=0$ has a nontrivial solution when $\rho = \pm k_i(z)$ or $\nu = \pm k_i(z)$. We exclude these $z$ as they cause difficulties in later theorems.
\begin{prop}\label{exclusion_set_L}
Define the set $\mathcal{L}\subset \mathbb{C}$ by
\[\mathcal{L}:=\{z\in \mathbb{C}: \exists i,j \in \{1,\dots,N\}, i\neq j \text{ such that } k_i^2(z)=k_j^2(z)\text{ or } k_i(z)=0\}.\]
Then $P_z(k_i(z),\nu)\neq 0$ and $P_z(\rho,k_i(z))\neq 0$ for all $\rho, \nu \in \mathbb{C}$ and $i \in \{1,\dots N\}$ if and only if $z \notin \mathcal{L}$.
\end{prop}

\begin{proof}
For $\rho,\nu \in \mathbb{C}$ and $i\in \{1,\dots, N\}$ we have that
\begin{align*}
P_z(\rho,k_i(z))&=-4 c_i(z)\,k_i^2(z)\prod_{\substack{j=1,\\ j\not=i}}^N (k_j^2(z)-\rho^2)(k_j^2(z)-k_i^2(z)),\\    
P_z(k_i(z),\nu)&=-4 c_i(z)\,k_i^2(z)\prod_{\substack{j=1,\\ j\not=i}}^N (k_j^2(z)-k_i^2(z))(k_j^2(z)-\nu^2).
\end{align*}
Clearly these are nonzero if and only if $z\notin \mathcal{L}$. 
\end{proof}

Consequently, for $z\notin \mathcal{L}$ we have that $P_z(\rho,\nu)=0$ is equivalent to 
\begin{equation}\label{char_pol_red}
Q_z(\rho,\nu):= (z+\alpha) - \sum_{i=1}^N \frac{4 c_i(z)\,k_i^2(z)}{(k_i^2(z)-\rho^2)(k_i^2(z)-\nu^2)}=0.
\end{equation}

We now want to use solutions of the PDE \eqref{PDE} to construct an eigenvector $q$ which solves \eqref{char_eq}. Unfortunately the set of the roots $\mathcal{N}(P_z) :=\{(\rho,\nu)\in\mathbb{C}\times\mathbb{C}: P_z(\rho, \nu)=0\}$ of the polynomial $P_z$ is uncountable. So we restrict ourselves to finite linear combinations of exponential solutions. For these finite linear combinations we can construct an explicit condition for solving \eqref{char_eq} that only uses values of $q$ at the boundary of $\bar{\Omega}$. In the next theorem we drop the $z$-dependence of the operators $K_i, L_i$ for clarity.

\begin{thm}\label{equiv_thm}
Let $z\in \mathbb{C}\setminus\{-\alpha\}$ such that $z \notin \mathcal{L}$. Let $q$ be a finite linear combination of exponential solutions of the PDE \eqref{PDE}, i.e. $q=\sum_{(\rho,\nu)\in V} q_{(\rho,\nu)}$, with $V$ a finite subset of $\mathcal{N}(P_z)$ and $q_{(\rho,\nu)}(x,y)=\gamma_{(\rho,\nu)}e^{\rho x}e^{\nu y}$, where $\gamma_{(\rho,\nu)}$ are some constants in $\mathbb{C}$. 

Then $\Delta(z)q=0$ if and only if 
\begin{equation} \label{boundary_cond}
\sum_{i=1}^N \sum_{(\rho,\nu)\in V} \frac{(K_i L_i - L_i K_i)q_{(\rho,\nu)}}{(k_i^2(z)-\rho^2)(k_i^2(z)-\nu^2)}=0.
\end{equation}
\end{thm}
\begin{proof} By definition, $q$ is a solution of the PDE \eqref{PDE}. From the definition of the operator $L_i$, for $z\notin \mathcal{L}$ we have that
\[
\frac{L_i q_{(\rho,\nu)}}{(k_i^2(z)-\rho^2)(k_i^2(z)-\nu^2)}=q_{(\rho,\nu)}.
\]
Hence we obtain that
\begin{align*}
\Delta(z)q &= (z+\alpha)q-\sum_{i=1}^N K_i(z)q\\
&=(z+\alpha)q-\sum_{i=1}^N \sum_{(\rho,\nu)\in V} \frac{K_i L_i q_{(\rho,\nu)}}{(k_i^2(z)-\rho^2)(k_i^2(z)-\nu^2)}\\
&=\sum_{(\rho,\nu)\in V}\left((z+\alpha) - \sum_{i=1}^N \frac{4c_i(z)\,k_i^2(z)}{(k_i^2(z)-\rho^2)(k_i^2(z)-\nu^2)}\right)q_{(\rho,\nu)} \\
&\quad - \sum_{i=1}^N \sum_{(\rho,\nu)\in V} \frac{(K_i L_i - L_i K_i)q_{(\rho,\nu)}}{(k_i^2(z)-\rho^2)(k_i^2(z)-\nu^2)}\\
&=-\sum_{i=1}^N \sum_{(\rho,\nu)\in V} \frac{(K_i L_i - L_i K_i)q_{(\rho,\nu)}}{(k_i^2(z)-\rho^2)(k_i^2(z)-\nu^2)},
\end{align*}
where we used that $L_i K_i q=4k_i^2(z)q$ and \eqref{char_pol_red}. 
\end{proof}

We can further evaluate the condition in \eqref{boundary_cond} using integration by parts. The operator on the left hand side of \eqref{KL_eq} is a operator that acts on $q$ and its normal derivatives at the boundary of $\bar\Omega$. Hence, we will refer to \eqref{boundary_cond} as the boundary condition from this point forward.

For a general $q\in C^\infty(\bar\Omega)$ the following holds for $i=1,2,\dots,N$
\begin{equation}\label{KL_eq}
\left(K_i(z) L_i(z) - L_i(z) K_i(z)\right)q= -2c_i(z)k_i(z)B_i(z)q+c_i(z)e^{-k_i(z)(a+b)}C_i(z)q,
\end{equation}
where
\begin{align}\label{BC_B}
(B_i(z)q)(x,y):=&e^{-k_i(z)(a+x)}\left(\left( k_i(z)-\ppf{}{x} \right)q\right)(-a,y)\nonumber\\
&+e^{-k_i(z)(a-x)}\left(\left( k_i(z)+\ppf{}{x} \right)q\right)(a,y)\nonumber\\
&+e^{-k_i(z)(b+y)}\left(\left( k_i(z)-\ppf{}{y} \right)q\right)(x,-b)\nonumber\\
&+e^{-k_i(z)(b-y)}\left(\left( k_i(z)+\ppf{}{y} \right)q\right)(x,b)
\end{align}
and
\begin{align}\label{BC_C}
(C_i(z)q)(x,y):=&e^{-k_i(z)(x+y)}\left(\left( k_i(z)-\ppf{}{x} \right)\left( k_i(z)-\ppf{}{y} \right)q\right)(-a,-b)\nonumber\\
&+e^{-k_i(z)(x-y)}\left(\left( k_i(z)-\ppf{}{x} \right)\left( k_i(z)+\ppf{}{y} \right)q\right)(-a,b)\nonumber\\
&+e^{k_i(z)(x-y)}\left(\left( k_i(z)+\ppf{}{x} \right)\left( k_i(z)-\ppf{}{y} \right)q\right)(a,-b)\nonumber\\
&+e^{k_i(z)(x+y)}\left(\left( k_i(z)+\ppf{}{x} \right)\left( k_i(z)+\ppf{}{y} \right)q\right)(a,b).
\end{align}

Finding eigenvectors with an arbitrary number of exponentials is still hard. However, we can show that if we have such an eigenvector, then a subset of at most $4N(N+1)$ exponentials form also an eigenvector. To prove this, we first need the following equivalence relation.


\begin{defin}
For every $z\in\mathbb{C},$ define the $\sim_z$ relation on $\mathbb{C}$ as follows: 
For $\rho, \nu \in \mathbb{C},$
\begin{equation}
\rho \sim_z \nu \text{ if and only if } P_z(\rho, \nu)=0 \text{ or } \rho^2 = \nu^2.
\end{equation}
\end{defin}
\begin{prop}\label{equivalence_class_Koen}
Let $z \notin \mathcal{L}$, then the relation $\sim_z$ defines an equivalence relation on $\mathbb{C}.$
\end{prop}
\begin{proof}
By definition $\sim_z$ is reflexive and by the symmetry of $P_z$, $\sim_z$ is also symmetric. 

Let $z \notin \mathcal{L}$. From Proposition \ref{exclusion_set_L}, we have that $P_z(\rho,\nu)=0$ if and only if $Q_z(\rho,\nu)=0$. Using equation \eqref{char_pol_red} we can deduce that 
\begin{equation}\label{rho_nu_q}
\begin{split}
(\rho^2-\nu^2)Q_z(\rho,\nu)& = (z+\alpha)(\rho^2-\nu^2)-(\rho^2-\nu^2) \sum_{i=1}^N \frac{4c_i(z)\,k_i^2(z)}{(k_i^2(z)-\rho^2)(k_i^2(z)-\nu^2)}\\
&=(z+\alpha)\rho^2 - \sum_{i=1}^N \frac{4c_i(z)\,k_i^2(z)}{(k_i^2(z)-\rho^2)} - (z+\alpha)\nu^2 + \sum_{i=1}^N\frac{4c_i(z)\,k_i^2(z)}{(k_i^2(z)-\nu^2)}.
\end{split}
\end{equation}
Let $\rho_1 \sim_z \rho_2$, $\rho_2 \sim_z \rho_3$. Due to equation \eqref{rho_nu_q}
\[(\rho_1^2-\rho_3^2)Q_z(\rho_1,\rho_3) = (\rho_1^2-\rho_2^2)Q_z(\rho_1,\rho_2) + (\rho_2^2-\rho_3^2)Q_z(\rho_2,\rho_3) = 0.\]
We conclude that $\rho_1 \sim_z \rho_3$ and hence $\sim_z$ is transitive. \end{proof}
For every $\nu \in \mathbb{C}$, we can construct an equivalence class $[\nu]_z = \{\rho \in \mathbb{C}| \rho \sim_z \nu\}$. Using this equivalence relation we can partition the null-space of $P_z$ into the following $E_{\nu,z}$ sets.

\begin{defin}
For every $z,\nu \in \mathbb{C}$, define the set $E_{\nu,z} \subset \mathcal{N}(P_z)$ as
\begin{equation}
E_{\nu,z} := \{(\rho_1,\rho_2) \in \mathcal{N}(P_z) | \rho_1 \sim_z \nu \sim_z \rho_2\}.
\end{equation}
\end{defin}

\begin{prop}[Partition principle]\label{partition_principle}
Let $z \notin \mathcal{L}$ and $\nu \in \mathbb{C}$. For all $(\rho_1,\nu_1) \in E_{\nu,z}$ and $(\rho_2,\nu_2)\in \mathcal{N}(P_z)\setminus E_{\nu,z}$, $\rho_1^2 \neq \rho_2^2$ and $\nu_1^2 \neq \nu_2^2$.
\end{prop}
\begin{proof}
Let $(\rho_1,\nu_1) \in E_{\nu,z}$ and $(\rho_2,\nu_2)\in \mathcal{N}(P_z)\setminus E_{\nu,z}$. Suppose $\rho_1^2 = \rho_2^2$, then $\nu \sim_z \rho_1 \sim_z \rho_2 \sim_z \nu_2$ and hence by Proposition \ref{equivalence_class_Koen}, $(\rho_2,\nu_2)\in E_{\nu,z}$. A similar reasoning holds when $\nu_1^2 \neq \nu_2^2$, so we have proven this statement by contradiction. \end{proof}

This property makes us able to split the boundary condition \eqref{boundary_cond} into multiple independent conditions corresponding to a single set $E_{\nu,z}$. Note that the above property does not hold for any non-empty, proper subset of $E_{\nu,z}$.

\begin{prop}
Suppose $z\notin \mathcal{L}$ and that $q$ is a finite linear combination of exponential solutions as in Theorem \ref{equiv_thm}, which solves $\Delta(z)q=0$. Then for all $\nu_1\in\mathbb{C}$ for which there exists a $\rho_1\in\mathbb{C}$ such that $(\rho_1,\nu_1)\in V$, with $V$ a finite subset of $\mathcal{N}(P_z)$, 
\begin{equation}
\sum_{i=1}^N \sum_{(\rho,\nu)\in E_{\nu_1,z}} \frac{B_i(z) q_{(\rho,\nu)}}{(k_i^2(z)-\rho^2)(k_i^2(z)-\nu^2)}=0.
\end{equation}
\end{prop}
\begin{proof}
Let $(\rho_1,\nu_1)\in V$ and $q$ as in Theorem \ref{equiv_thm}. Using \eqref{KL_eq}-\eqref{BC_C} we find that for all $(\rho,\nu)\in V$ and $i=1,\dots,N$
\begin{equation*}
(K_i L_i-L_i K_i)q_{(\rho,\nu)}=-2c_i(z)k_i(z)B_i(z)q_{(\rho,\nu)}+c_i(z)e^{-k_i(z)(a+b)}C_i(z)q_{(\rho,\nu)}\end{equation*}
with
\begin{equation}
\begin{split}
\left(B_i(z) q_{(\rho,\nu)}\right)(x,y)=& \gamma_{(\rho,\nu)}\left[e^{-k_i(z)(a+x) -\rho a + \nu y}(\rho-k_i(z))\right.\\
&-e^{-k_i(z)(a-x)+\rho a + \nu y}(\rho+k_i(z))\\
&+e^{-k_i(z)(b+y)+\rho x-\nu b}(\nu-k_i(z))\\
&\left.-e^{-k_i(z)(b-y)+\rho x+\nu b}(\nu+k_i(z))\right],\\[5pt]
\left(C_i(z) q_{(\rho,\nu)}\right)(x,y)=&\gamma_{(\rho,\nu)}\left[e^{-k_i(z)(x+y)-\rho a - \nu b}(\rho-k_i(z))(\nu-k_i(z))\right.\\
&-e^{-k_i(z)(x-y)-\rho a + \nu b}(\rho-k_i(z))(\nu+k_i(z))\\
&-e^{k_i(z)(x-y)+\rho a - \nu b}(\rho+k_i(z))(\nu-k_i(z))\\
&\left.+e^{k_i(z)(x+y)+\rho a + \nu b}(\rho+k_i(z))(\nu+k_i(z))\right].
\end{split}
\end{equation}
We note that $B_i(z) q_{(\rho,\nu)}$ is a linear combination of exponentials $e^{\pm k_i(z) x}e^{\nu y}$ and $e^{\rho x}e^{\pm k_i(z) y}$ and that $C_i(z) q_{(\rho,\nu)}$ of $e^{\pm k_i(z) x}e^{\pm k_i(z) y}$.

As $z\notin \mathcal{L}$, we get by Proposition \ref{exclusion_set_L} that for all $(\rho,\nu)\in V$, $\rho,\nu \neq \pm k_i(z)$ for $i\in \{1,\dots, N\}$. Furthermore, as $(\rho_1,\nu_1)\in E_{\nu_1,z}$, by Proposition \ref{partition_principle}, for all $(\rho_2,\nu_2) \in V\setminus E_{\nu_1,z}$, $\rho_1^2 \neq  \rho_2^2$ and $\nu_1^2 \neq \nu_2^2$. Hence the elements of the set 
\[\{e^{\pm k_i(z) x}e^{\nu_1 y}, e^{\rho_1 x}e^{\pm k_i(z) y},e^{\pm k_i(z) x}e^{\nu_2 y},e^{\rho_2 x}e^{\pm k_i(z) y}, e^{\pm k_i(z) x}e^{\pm k_i(z) y} \mid i= 1,\dots, N \}\] 
are linearly independent.

We conclude that the terms
\begin{align*}
\sum_{i=1}^N \sum_{(\rho,\nu)\in E_{\nu_1,z}} \frac{B_i(z) q_{(\rho,\nu)}}{(k_i^2(z)-\rho^2)(k_i^2(z)-\nu^2)},\\
\sum_{i=1}^N \sum_{(\rho,\nu)\in V\setminus E_{\nu_1,z}} \frac{B_i(z) q_{(\rho,\nu)}}{(k_i^2(z)-\rho^2)(k_i^2(z)-\nu^2)},\\
\sum_{i=1}^N \sum_{(\rho,\nu)\in V} \frac{C_i(z) q_{(\rho,\nu)}}{(k_i^2(z)-\rho^2)(k_i^2(z)-\nu^2)}
\end{align*}
are linearly independent. As $q$ satisfies the boundary conditions \eqref{boundary_cond}, these should vanish. 
\end{proof}

Generically, a polynomial of degree $2N$ has $2N$ distinct roots. As we show in the proposition below, this implies a generic representation of $E_{\nu,z}$.
\begin{prop}\label{equiv_class_size}
Let $z \notin \mathcal{L}, \nu \in \mathbb{C}$ and suppose that the equivalence class $[\nu]_z$ has $2(N+1)$ distinct elements $\pm \rho_1, \dots, \pm \rho_{N+1}$. Then $E_{\nu,z} = \{(\pm \rho_i, \pm \rho_j) \mid i,j\in \{1,\dots,N+1\}, i\neq j\}$ and $\rho_i\neq 0$ for $i \in \{1,\dots,N+1\}$.
\end{prop}
\begin{proof}
For given $\nu\in\mathbb{C}$, $P_z(\rho,\nu)$ is a polynomial of order $2N$ in $\rho$. So if $[\nu]_z$ has 
$2(N+1)$ distinct elements, then $P_z(\rho,\nu)=0$ must have $2N$ distinct solutions $\rho = \pm\rho_1, \dots, \pm \rho_{N}$, which are not equal to $\pm \nu$, so we can define $\rho_{N+1}:=\nu$. Furthermore as $\rho_i$ must be distinct from $-\rho_i$, this implies that $\rho_i\neq 0$ for $i \in \{1,\dots,N\}$. 
\end{proof}

The sum of exponentials corresponding to $E_{\nu,z}$ can equivalently be expressed as
\begin{equation}
\begin{split}
q_{E_{\nu,z}}(x,y)=\sum_{i,j=1}^{N+1} &\left[d^{ee}_{ij} \cosh(\rho_i x)\cosh(\rho_j y)+d^{eo}_{ij} \cosh(\rho_i x)\sinh(\rho_j y)\right.\\
+&\left.d^{oe}_{ij} \sinh(\rho_i x)\cosh(\rho_j y)+d^{oo}_{ij} \sinh(\rho_i x)\sinh(\rho_j y)\right],
\end{split}\label{eq:eigenvector_def}
\end{equation}
where we require that $d_{ii}=0$ for $i=\{1,\dots,N+1\}$. The coefficients $d^{ee}_{ij},d^{eo}_{ij},d^{oe}_{ij},d^{oo}_{ij}$ form the matrices $D^{ee}, D^{eo}, D^{oe}, D^{oo} \in \mathbb{C}^{(N+1)\times (N+1)}$ with a zero diagonal. The superscripts $e,o$ refer to coefficients of even and odd functions of $x$ and $y$, respectively.

We define the matrices $S^{e},S^{o}\in \mathbb{C}^{N\times (N+1)}$ with elements
\begin{equation}\label{S_matrices}
\begin{split}
S^{e}_{ij}(r,\nu,z)&:= \frac{k_i(z) \cosh(\rho_j(\nu,z)r) + \rho_j(\nu,z) \sinh(\rho_j(\nu,z) r)}{k_i^2(z)-\rho_j^2(\nu,z)}\\
S^{o}_{ij}(r,\nu,z)&:= \frac{k_i(z) \sinh(\rho_j(\nu,z)r) + \rho_j(\nu,z) \cosh(\rho_j(\nu,z) r)}{k_i^2(z)-\rho_j^2(\nu,z)}
\end{split}
\end{equation}
for $i \in \{1,\dots,N\}, j \in \{1,\dots,N+1\}$. The superscripts $e,o$ refer to even and odd functions in $\rho_j$, respectively.

\begin{prop}
Let $z \notin \mathcal{L}, \nu_1 \in \mathbb{C}$ and suppose that the equivalence class $[\nu_1]_z$ has $2(N+1)$ distinct elements. Then 
\begin{equation}\label{Bq}
\sum_{l=1}^N \sum_{(\rho,\nu)\in E_{\nu_1,z}} \frac{B_l(z) q_{(\rho,\nu)}}{(k_l^2(z)-\rho^2)(k_l^2(z)-\nu^2)}=0
\end{equation}
implies that
\begin{equation}\label{Cq}
\sum_{l=1}^N \sum_{(\rho,\nu)\in E_{\nu_1,z}} \frac{C_l(z) q_{(\rho,\nu)}}{(k_l^2(z)-\rho^2)(k_l^2(z)-\nu^2)}=0.
\end{equation}
\end{prop}
\begin{proof}
For this proof we drop the dependency on $z$ and $\nu_1$. We do some calculations beforehand using \eqref{BC_B} and \eqref{BC_C}:
\small{
\begin{align*}
B_l\cosh(\rho x)\sinh(\nu y)= &-2e^{-k_l a}\cosh(k_l x)\sinh(\nu y)(k_l \cosh(\rho a)+ \rho \sinh(\rho a))\\
&-2e^{-k_l b}\cosh(\rho x)\sinh(k_l y)(k_l \sinh(\nu b)+ \nu \cosh(\nu b)),\\
C_l\cosh(\rho x)\sinh(\nu y)=  & 4\cosh(k_l x)\sinh(k_l y)\left(k_l \cosh(\rho a)+ \rho \sinh(\rho a))(k_l \sinh(\nu b)+ \nu \cosh(\nu b)\right),
\end{align*}
\begin{align*}
B_l\cosh(\rho x)\cosh(\nu y)= &-2e^{-k_l a}\cosh(k_l x)\cosh(\nu y)(k_l \cosh(\rho a)+ \rho \sinh(\rho a))\\
&-2e^{-k_l b}\cosh(\rho x)\cosh(k_l y)(k_l \cosh(\nu b)+ \nu \sinh(\nu b)),\\
C_l\cosh(\rho x)\cosh(\nu y)=  & 4\cosh(k_l x)\cosh(k_l y)\left(k_l \cosh(\rho a)+ \rho \sinh(\rho a))(k_l \cosh(\nu b)+ \nu \sinh(\nu b)\right),
\end{align*}
\begin{align*}
B_l\sinh(\rho x)\sinh(\nu y)= &-2e^{-k_l a}\sinh(k_l x)\sinh(\nu y)(k_l \sinh(\rho a)+ \rho \cosh(\rho a))\\
&-2e^{-k_l b}\sinh(\rho x)\sinh(k_l y)(k_l \sinh(\nu b)+ \nu \cosh(\nu b)),\\
C_l\sinh(\rho x)\sinh(\nu y)=  & 4\sinh(k_l x)\sinh(k_l y)\left(k_l \sinh(\rho a)+ \rho \cosh(\rho a))(k_l \sinh(\nu b)+ \nu \cosh(\nu b)\right),
\end{align*}
\begin{align*}
B_l\sinh(\rho x)\cosh(\nu y)= &-2e^{-k_l a}\sinh(k_l x)\cosh(\nu y)(k_l \sinh(\rho a)+ \rho \cosh(\rho a))\\
&-2e^{-k_l b}\sin(\rho x)\cosh(k_l y)(k_l \cosh(\nu b)+ \nu \sinh(\nu b)),\\
C_l\sinh(\rho x)\cosh(\nu y)=  & 4\sinh(k_l x)\cosh(k_l y)\left(k_l \sinh(\rho a)+ \rho \cosh(\rho a))(k_l \cosh(\nu b)+ \nu \sinh(\nu b)\right),
\end{align*}
}
\normalsize
hold for $l\in \{1,\dots,N\}$. Now we expand \eqref{Bq} and \eqref{Cq} as 
\small
\allowdisplaybreaks
\begin{align*}
0=&\sum_{l=1}^N \sum_{(\rho,\nu)\in E_{\nu_1,z}} \frac{B_l q_{(\rho,\nu)}}{(k_l^2(z)-\rho^2)(k_l^2(z)-\nu^2)}\\
=& \sum_{l=1}^N \sum_{j=1}^{N+1} 2 e^{-k_l a}\frac{\cosh(k_l x)\cosh(\rho_j y)}{k_l^2-\rho_j^2} \Bigl(\sum_{i=1}^{N+1} d^{ee}_{ij}\Bigl(\frac{k_l\cosh(\rho_i a) + \rho_i \sinh(\rho_i a)}{k_l^2-\rho_i^2}\Bigr)\Bigr)\\
&+\sum_{l=1}^N \sum_{i=1}^{N+1} 2 e^{-k_l b}\frac{\cosh(\rho_i x)\cosh(k_l y)}{k_l^2-\rho_i^2} \Bigl(\sum_{j=1}^{N+1} d^{ee}_{ij}\Bigl(\frac{k_l\cosh(\rho_j b)+\rho_j \sinh(\rho_j b)}{k_l^2-\rho_j^2}\Bigr)\Bigr)\\
& +\sum_{l=1}^N \sum_{j=1}^{N+1} 2 e^{-k_l a}\frac{\cosh(k_l x)\sinh(\rho_j y)}{k_l^2-\rho_j^2} \Bigl(\sum_{i=1}^{N+1} d^{eo}_{ij}\Bigl(\frac{k_l\cosh(\rho_i a)+\rho_i \sinh(\rho_i a)}{k_l^2-\rho_i^2}\Bigr)\Bigr)\\
& +\sum_{l=1}^N \sum_{i=1}^{N+1} 2 e^{-k_l b}\frac{\cosh(\rho_i x)\sinh(k_l y)}{k_l^2-\rho_i^2} \Bigl(\sum_{j=1}^{N+1} d^{eo}_{ij}\Bigl(\frac{k_l\sinh(\rho_j b)+\rho_j \cosh(\rho_j b)}{k_l^2-\rho_j^2}\Bigr)\Bigr)\\
& +\sum_{l=1}^N \sum_{j=1}^{N+1} 2 e^{-k_l a}\frac{\sinh(k_l x)\cosh(\rho_j y)}{k_l^2-\rho_j^2} \Bigl(\sum_{i=1}^{N+1} d^{oe}_{ij}\Bigl(\frac{k_l\sinh(\rho_i a) + \rho_i \cosh(\rho_i a)}{k_l^2-\rho_i^2}\Bigr)\Bigr)\\
& +\sum_{l=1}^N \sum_{i=1}^{N+1} 2 e^{-k_l b}\frac{\sinh(\rho_i x)\cosh(k_l y)}{k_l^2-\rho_i^2} \Bigl(\sum_{j=1}^{N+1} d^{oe}_{ij}\Bigl(\frac{k_l\cosh(\rho_j b)+\rho_j \sinh(\rho_j b)}{k_l^2-\rho_j^2}\Bigr)\Bigr)\\
& +\sum_{l=1}^N \sum_{j=1}^{N+1} 2 e^{-k_l a}\frac{\sinh(k_l x)\sinh(\rho_j y)}{k_l^2-\rho_j^2} \Bigl(\sum_{i=1}^{N+1} d^{oo}_{ij}\Bigl(\frac{k_l\sinh(\rho_i a)+\rho_i \cosh(\rho_i a)}{k_l^2-\rho_i^2}\Bigr)\Bigr)\\
& +\sum_{l=1}^N \sum_{i=1}^{N+1} 2 e^{-k_l b}\frac{\sinh(\rho_i x)\sinh(k_l y)}{k_l^2-\rho_i^2} \Bigl(\sum_{j=1}^{N+1} d^{oo}_{ij}\Bigl(\frac{k_l\sinh(\rho_j b)+\rho_j \cosh(\rho_j b)}{k_l^2-\rho_j^2}\Bigr)\Bigr)\end{align*}
and
\allowdisplaybreaks
\begin{align*}
&\sum_{l=1}^N \sum_{(\rho,\nu)\in E_{\nu_1,z}} \frac{C_l q_{(\rho,\nu)}}{(k_l^2(z)-\rho^2)(k_l^2(z)-\nu^2)}\\
=& \sum_{l=1}^N 4\cosh(k_l x)\cosh(k_l y) \Bigl(\sum_{i,j=1}^{N+1} d^{ee}_{ij}\Bigl(\frac{k_l\cosh(\rho_i a)+\rho_i \sinh(\rho_i a)}{k_l^2-\rho_i^2}\Bigr)\Bigl(\frac{k_l\cosh(\rho_j b)+\rho_j \sinh(\rho_j b)}{k_l^2-\rho_j^2}\Bigr)\Bigr)\\
&+ \sum_{l=1}^N 4 \cosh(k_l x)\sinh(k_l y) \Bigl(\sum_{i,j=1}^{N+1} d^{eo}_{ij}\Bigl(\frac{k_l\cosh(\rho_i a)+\rho_i \sinh(\rho_i a)}{k_l^2-\rho_i^2}\Bigr)\Bigl(\frac{k_l\sinh(\rho_j b)+\rho_j \cosh(\rho_j b)}{k_l^2-\rho_j^2}\Bigr)\Bigr)\\
&+ \sum_{l=1}^N 4 \sinh(k_l x)\cosh(k_l y) \Bigl(\sum_{i,j=1}^{N+1} d^{oe}_{ij}\Bigl(\frac{k_l\sinh(\rho_i a)+\rho_i \cosh(\rho_i a)}{k_l^2-\rho_i^2}\Bigr)\Bigl(\frac{k_l\cosh(\rho_j b)+\rho_j \sinh(\rho_j b)}{k_l^2-\rho_j^2}\Bigr)\Bigr)\\
&+ \sum_{l=1}^N 4 \sinh(k_l x)\sinh(k_l y) \Bigl(\sum_{i,j=1}^{N+1} d^{oo}_{ij}\Bigl(\frac{k_l\sinh(\rho_i a)+\rho_i \cosh(\rho_i a)}{k_l^2-\rho_i^2}\Bigr)\Bigl(\frac{k_l\sinh(\rho_j b)+\rho_j \cosh(\rho_j b)}{k_l^2-\rho_j^2}\Bigr)\Bigr).
\end{align*}
\normalsize
Since $z \notin \mathcal{L}$, we have by Proposition \ref{exclusion_set_L} that $\rho_i \neq \pm k_l$ and $k_p \neq \pm k_l$ for $i\in \{1,\dots, N+1\}$ and $p,l \in \{1,\dots, N\}$, where $p\neq l$. Since $[\nu]_z$ has $2(N+1)$ elements we have by Proposition~\ref{equiv_class_size} that $\rho_i \neq \pm \rho_j \neq 0$ for $i,j \in \{1,\dots, N+1\}$, where $i\neq j$. Hence all the terms above of cosine hyperbolic and sine hyperbolic in $x$ and $y$ are linearly independent and hence all the sums on the right have to be zero to satisfy the conditions. 

Using the matrices defined in \eqref{S_matrices}, the necessary and sufficient conditions for which \eqref{Bq} holds are
\begin{equation}\label{matrix_equations}
\begin{split}
S^{e}(a)D^{ee}&=S^{e}(b)(D^{ee})^{T}=S^{e}(a)D^{eo}=S^{o}(b)(D^{eo})^{T}=O\\
S^{o}(a)D^{oe}&=S^{e}(b)(D^{oe})^{T}=S^{o}(a)D^{oo}=S^{o}(b)(D^{oo})^{T}=O,
\end{split}
\end{equation}
where $O$ is the $N\times (N+1)$-zero matrix. The necessary and sufficient conditions for which \eqref{Cq} holds are
\begin{equation}
\begin{split}
S^{e}(a)D^{ee}(S^{e}(b))^{T}&=S^{e}(a)D^{eo}(S^{o}(b))^{T}=O\\
S^{o}(a)D^{oe}(S^{e}(b))^{T}&=S^{o}(a)D^{oo}(S^{o}(b))^{T}=O,
\end{split}
\end{equation}
where $O$ is the $N\times N$-zero matrix. Hence we conclude that \eqref{Bq} implies \eqref{Cq}.
\end{proof}

\begin{cor}\label{cor:final}
Let $z \notin \mathcal{L}, \nu_1 \in \mathbb{C}$ and suppose that the equivalence class $[\nu_1]_z$ has $2(N+1)$ distinct elements. Then $\Delta(z)q_{E_{\nu_1,z}}=0$ if and only if
\begin{equation}
\sum_{l=1}^N \sum_{(\rho,\nu)\in E_{\nu_1,z}} \frac{B_l(z) q_{(\rho,\nu)}}{(k_l^2(z)-\rho^2)(k_l^2(z)-\nu^2)}=0.
\end{equation}
\end{cor}

We can conclude that finding eigenvectors comes down to finding non-trivial solutions to the matrix equations \eqref{matrix_equations}. The problem is that for $D^{ee}$ to be non-zero, there are $2N(N+1)$ equations and $N(N+1)$ unknowns in $D^{ee}$ and $\nu_1$ and $z$, in total $N(N+1)+2$ unknowns. When $N=1$ then the number of equations and unknowns are the same. This special case will be treated separately in the next section. When $N\geq 2$ then the number of conditions is larger than the number of unknowns. However, we can reduce the amount of equations to some extent. To illustrate this, we consider $N=2$ and distinguish two cases: $\rank(D^{ee})=1$ and $\rank(D^{ee})\geq 2$.

When $\rank(D^{ee})=1$, it means that we can write $D^{ee}=d^a (d^b)^T$, where $d^a,d^b \in \mathbb{C}^3$ are non-trivial solutions of $S^{e}(a)d^a=0, S^{e}(b)d^b=0$. To satisfy the condition that $D^{ee}$ has a zero diagonal, one of $d^a$ or $d^b$ has to have two zero elements. Without loss of generality, suppose that $d^a= (1\; 0\; 0)^T$. Then $S^{e}(a)d^a=0$ implies that $S^{e}_{11}(a)=0$ and $S^{e}_{21}(a)=0$, or equivalently,
\begin{equation}
\begin{split}
&k_1(z) \cosh(\rho_1 a) + \rho_1 \sinh(\rho_1 a)=0,\\
&k_2(z) \cosh(\rho_1 a) + \rho_1 \sinh(\rho_1 a)=0.
\end{split}\label{eq:boundary_N2}
\end{equation}
Subtracting these equations and using that $k_1(z)\neq k_2(z)$, we get that $\cosh(\rho_1 a)=0$. Substituting this back into \eqref{eq:boundary_N2} implies that $\rho_1 \sinh(\rho_1 a)=0$, which is not possible simultaneously with $\cosh(\rho_1 a)=0$. Hence $\rank(D^{ee})$ must be at least 2.

Due to the Sylvester's rank inequality, $\rank(S^{e}(a))+\rank(D^{ee})\leq 3$ and $\rank(S^{e}(b))+\rank(D^{ee})\leq 3$. Hence $\rank(D^{ee})\geq 2$ implies that $\rank(S^{e}(a)),\rank(S^{e}(b))\leq 1$. In this case it is then possible to construct an explicit solution for $D^{ee}$, so the condition $\rank(S^{e}(a))$, $\rank(S^{e}(b))\leq 1$ is also sufficient.

Requiring that $\rank(S^{e}(a))=1$, i.e., the two rows of $S^{e}(a)$ are linearly dependent, gives the following conditions
\begin{align*}
\eta_1 S^{e}_{11}(a) + \eta_2 S^{e}_{12}(a) = 0\\
\eta_1 S^{e}_{21}(a) + \eta_2 S^{e}_{22}(a) = 0\\
\eta_1 S^{e}_{31}(a) + \eta_2 S^{e}_{32}(a) = 0\\
\eta_1^2+\eta_2^2 = 1.
\end{align*}
Under the same assumption on $S^{e}(b)$, for each of the matrices $S^{e}(a),S^{e}(b)$ there are 4 conditions and we have in total 6 unknowns, $\nu_1,z$ and the $\eta_1,\eta_2$ constants for each matrix. Hence this system is overdetermined and therefore has no generic solution.

However, when $a=b$ the matrices $S^{e}(a)$ and $S^{e}(b)$ are the same and this means 4 conditions and 4 unknowns. Therefore, with a Newton method we can find $\nu_1,z$ such that $\rank(S^{e}(a))=1$ and from that we can explicitly find a non-trivial solution $D^{ee}$ of the form
\[D^{ee}=\begin{pmatrix}
0 & - S^e_{13} & S^e_{12}\\
S^e_{13} & 0 & -S^e_{11}\\
-S^e_{12} & S^e_{11} & 0
\end{pmatrix}.\]
This works similarly for $\rank(S^{o}(a))=1$ and $D^{oo}$. Obtaining however a non-trivial $D^{eo}$ or $D^{oe}$ still requires 8 conditions to be satisfied and it is thus non-generic.

In conclusion, there are only generic eigenvectors which are a (finite) sum of exponentials if $a=b$. 
Otherwise, generic eigenvectors for $a\neq b$ (and maybe also some for $a=b$) are solutions to \eqref{char_eq} which are not a sum of exponentials. In general, there is no guarantee that solutions of these integral equations can be expressed analytically. 
\section{Single exponential connectivity}\label{sec:single_exp}
In this section, we consider the case when the connectivity kernel in \eqref{connectivity} is a single exponential, i.e., $N=1$. In contrast to $N\geq 2$, in this case we can find a complete characterisation of the spectrum. Using the results of Section~\ref{section spectrum}, we formulate the boundary value problem and give an analytic representation of the solution. After having described the spectrum of the linearized neural field equation, we solve the resolvent problem for this special case and using these results, we give an example of a Hopf bifurcation in the next section. For notational simplicity, we drop the subscripts of the operators. 

As shown in Section~\ref{section spectrum}, we can transform the integral equation \eqref{char_eq} into a PDE using the fact that $L(z)K(z)q = 4 c(z) k^2(z) q$. In the next theorem we state that, for $N=1$ the characteristic integral equation of the DDE is equivalent to a PDE with an additional (boundary) condition.
\begin{thm}\label{thm:equivalence}
Let $q\in Y$ such that $q_{xxyy}=q_{yyxx}\in C(\Omega)$ and let $z\in \mathbb{C}$ such that $z\neq -\alpha$ and $k(z)\neq 0$. Then we have the following equivalence
\begin{equation}\label{eq:equiv_N1}
\Delta(z) q=0  \Leftrightarrow \left\{L(z)\Delta(z)q=0 \text{ and } K(z)L(z)q=L(z)K(z)q \right\}.
\end{equation}
Moreover, for any $g\in Y$ we get that 
\begin{equation}\label{eq:equiv_N2}
\Delta(z) q = K g  \Leftrightarrow \left\{L(z)\Delta(z)q= 4 c(z)k^2(z) g \text{ and } K(z)L(z)q=L(z)K(z)q \right\}.
\end{equation}
\end{thm}
\begin{proof} Let $q,g\in Y$ such that $q_{xxyy}=q_{yyxx}\in C(\Omega)$ and let $z\in \mathbb{C}\setminus\{-\alpha\}$, such that $k(z)\neq 0$. The smoothness condition on $q$ implies that 
\begin{equation}
L(z)q = \left(k^2(z) - \frac{\partial^2}{\partial x^2} \right)\circ \left(k^2(z) -\frac{\partial^2}{\partial y^2} \right)q
\end{equation}
is well defined. Using $g\equiv 0$ in \eqref{eq:equiv_N2} gives \eqref{eq:equiv_N1}, so we only prove \eqref{eq:equiv_N2}. For the remaining part of this proof we will drop the dependency on $z$ for clarity.

First we apply the operator $L$ and $KL$ to $(\Delta q - Kg)$
\begin{align*}
L(\Delta q - Kg) &= (z+\alpha)Lq - LKq - LKg\\
&=  (z+\alpha)Lq - 4ck^2 q - 4ck^2 g,\\
KL(\Delta q - Kg) &= (z+\alpha)KL q - 4ck^2 Kq - 4ck^2 Kg\\
&= (z+\alpha)KL q + 4ck^2 \Delta q  - 4ck^2(z+\alpha) q - 4ck^2 Kg\\
&= (z+\alpha)(KL-LK)q + 4ck^2(\Delta q - Kg).
\end{align*}
Suppose that $\Delta q= Kg$. Then from the equations above we get that 
\[L\Delta q =  (z+\alpha)Lq - 4ck^2 q = 4ck^2 g
\]
and that $(KL-LK)q=0$. 

Conversely, suppose that $L\Delta q = 4ck^2 g$ and that $(KL-LK)q=0$. Then 
\[0=KL(\Delta q - Kg) = 4ck^2(\Delta q - Kg),\] 
and hence $\Delta q = Kg$ as $k\neq 0$.
\end{proof}

Note that, for the set of $q$ where $K(z)L(z)q=L(z)K(z)q= 4c(z) k^2(z) q$ holds, $K(z)$ has a two-sided inverse $\tfrac{1}{4c(z) k^2(z)} L(z)$. Using \eqref{KL_eq}, we can write $(K(z)L(z)-L(z)K(z))q$ in terms of derivatives of $q$ at the boundary of $\Omega$. In the next lemma we can see that, under some conditions, we can interpret the right-hand side of \eqref{eq:equiv_N1} as a boundary value problem with a Robin-type boundary condition.

\begin{lem}\label{normalBC}
Let $q\in Y$ such that $q_{xxyy}=q_{yyxx}\in C(\Omega)$. Then 
\begin{equation}\label{BC}
\left(k(z)+ \frac{\partial}{\partial n}\right)q(x,y)=0 \quad \forall (x,y)\in\partial\Omega
\end{equation}
implies that $(K(z)L(z)-L(z)K(z))q=0$, where $\frac{\partial}{\partial n}$ is the outward normal derivative to the boundary of $\Omega$. 

If $k(z)\neq 0$ and $q(x,y)=\phi(x)\psi(y)$, where $\phi\in C^2([-a,a])$, $\psi\in C^2([-b,b])$ and $\phi(x) \neq \bar{c} e^{\pm k(z) x}$, $\psi(y) \neq \bar{c} e^{\pm k(z) y}$ for all $\bar{c}\in \mathbb{C}$, then $(K(z)L(z)-L(z)K(z))q=0$ also implies \eqref{BC}.
\end{lem}
\begin{proof} From \eqref{KL_eq},we can write $K(z)L(z)-L(z)K(z)$ in terms of the operators $B(z)$ and $C(z)$ as
\begin{equation}\label{BCequiv}
    (K(z)L(z)-L(z)K(z))q=-2c(z)k(z)B(z)q+c(z) e^{-k(z)(a+b)}C(z)q=0.
\end{equation}
The first statement then immediately follows from the definition of $B(z)$ and $C(z)$ in \eqref{BC_B} and \eqref{BC_C}, respectively. 

Conversely, assume that $q(x,y)=\phi(x)\psi(y)$, with $\phi$ and $\psi$ as in the second statement of the lemma. Then we can write $B(z)q$ and $C(z)q$ as
\begin{align*}
(B(z)q)(x,y)=& \psi(y)e^{-k(z)(a+x)}\left(\left( k(z)-\ppf{}{x} \right)\phi\right)(-a)\\
&+\psi(y)e^{-k(z)(a-x)}\left(\left( k(z)+\ppf{}{x} \right)\phi\right)(a)\\
&+\phi(x)e^{-k(z)(b+y)}\left(\left( k(z)-\ppf{}{y} \right)\psi\right)(-b)\\
&+\phi(x)e^{-k(z)(b-y)}\left(\left( k(z)+\ppf{}{y} \right)\psi\right)(b)
\end{align*}
and
\begin{align*}
(C(z)q)(x,y)=&e^{-k(z)x}\left(\left( k(z)-\ppf{}{x}\right)\phi \right)(-a)\; e^{-k(z)(y)}\left(\left( k(z)-\ppf{}{y} \right)\psi\right)(-b)\\
&+e^{-k(z)x}\left(\left( k(z)-\ppf{}{x}\right)\phi \right)(-a)\; e^{k(z)y}\left(\left( k(z)+\ppf{}{y} \right)\psi\right)(b)\\
&+e^{k(z)x}\left(\left( k(z)+\ppf{}{x}\right)\phi \right)(a)\; e^{-k(z)y}\left(\left( k(z)-\ppf{}{y} \right)\psi\right)(-b)\\
&+e^{k(z)x}\left(\left( k(z)+\ppf{}{x}\right)\phi \right)(a)\; e^{k(z)y}\left(\left( k(z)+\ppf{}{y} \right)\psi\right)(b).
\end{align*}
Using the fact that $k(z)\neq 0$ and $\phi(x) \neq \bar{c} e^{\pm k(z) x}, \psi(y) \neq \bar{c} e^{\pm k(z) y}$ for all $\bar{c}\in \mathbb{C}$, we can reason by linear independence that each term of $B(z)q$ should vanish and hence \eqref{BC} holds.
\end{proof}

We can rewrite the PDE $L(z)\Delta(z) q = 0$ as
\begin{equation}\label{PDE_N1}
L(z)q= \frac{4c(z)k^2(z)}{z+\alpha} q.
\end{equation}
So $z$ is an eigenvalue of the original DDE, when $\frac{4c(z)k^2(z)}{z+\alpha}$ is an eigenvalue of $L(z)$ with a domain of functions $q$ which satisfy the boundary condition $K(z)L(z)q=L(z)K(z)q$.

\subsection{Eigenvalues and eigenvectors}
We can use Corollary \ref{cor:final} and the matrix equations of \eqref{matrix_equations} to find some eigenvalues with eigenvectors which are a sum of exponentials. But first we take a look at the set of resonances $\mathcal{L}$.

From the definition of $\mathcal{L}$ in Proposition \ref{exclusion_set_L} we see that $z\in \mathcal{L}$ reduces to $k(z)=0$ for $N=1$. When $k(z)=0$, i.e. $z=-\xi$, any solution $q$ to $\Delta(z) q=0$ is constant, as
\[(z+\alpha)q(r) = c(z)\int_\Omega q(r') dr'.\]
Hence $z=-\xi$ is an eigenvalue if and only if 
\[\xi-\alpha + 4ab\, c(-\xi)=0.\]

We can now characterize the eigenvalues $z$, i.e., those $z$ values for which $\Delta(z)q=0$ has a non-trivial solution.
\begin{thm}\label{thm:eigenvectors}
Let $z\in \mathbb{C}\setminus \{-\alpha\}$ such that $k(z)\neq 0$ and let $\nu,\rho \in\mathbb{C}$ such that $P_z(\rho,\nu)= 0$ with $\rho,\nu\neq 0$, where
\begin{equation}\label{char_pol_N1}
P_z(\rho, \nu)=-(z+\alpha)(k^2(z)-\rho^2)(k^2(z)-\nu^2)+4c(z) k^2(z).
\end{equation}

If $k(z) \cosh(\rho a) + \rho \sinh(\rho a) = k(z) \cosh(\nu b) + \rho \sinh(\nu b) =0$, then $z$ is an eigenvalue with the eigenvector $q(x,y) = \cosh(\rho x)\cosh(\nu y)$.

If $k(z) \sinh(\rho a) + \rho \cosh(\rho a) = k(z) \cosh(\nu b) + \rho \sinh(\nu b) =0$, then $z$ is an eigenvalue with the eigenvector $q(x,y) = \sinh(\rho x)\cosh(\nu y)$.

If $k(z) \cosh(\rho a) + \rho \sinh(\rho a) = k(z) \sinh(\nu b) + \rho \cosh(\nu b) =0$, then $z$ is an eigenvalue with the eigenvector $q(x,y) = \cosh(\rho x)\sinh(\nu y)$.

If $k(z) \sinh(\rho a) + \rho \cosh(\rho a) = k(z) \sinh(\nu b) + \rho \cosh(\nu b) =0$, then $z$ is an eigenvalue with the eigenvector $q(x,y) = \sinh(\rho x)\sinh(\nu y)$.
\end{thm}
\begin{proof}
We can directly apply Corollary \ref{cor:final}. More specifically the matrix equations \eqref{matrix_equations} hold. Without loss of generality, we set $D^{eo},D^{oe},D^{oo}=O$, leaving only $D^{ee}$ as a variable satisfying $S^e(a)D^{ee}=S^e(b)(D^{ee})^T=O$. For $N=1$, $D^{ee}$ has the following structure 
\[D^{ee} = \begin{pmatrix}
0 & d_{12}\\
d_{21} & 0
\end{pmatrix}.\]
Hence the equations in $S^e(a)D^{ee}=S^e(b)(D^{ee})^T=O$ decouple, so we can solve for $d_{12}$ and $d_{21}$ independently. Without loss of generality we can set $d_{21}=0$, which gives the following set of equations
\[(k(z) \cosh(\rho a) + \rho \sinh(\rho a))d_{12} = (k(z) \cosh(\nu b) + \nu\sinh(\nu b))d_{12} =0.\]
Set $d_{12}$ to an arbitrary non-zero complex value, leaving the remaining conditions in the theorem. From \eqref{eq:eigenvector_def}, we get that the eigenvector $q$ in this case has the form $q(x,y) = \cosh(\rho x)\cosh(\nu y)$. 

Choosing a different matrix from $D^{eo},D^{oe},D^{oo}$ to be nonzero, gives the remaining conditions in the theorem. 
\end{proof}
The two equations in the theorem above, together with the condition $P_z(\rho,\nu)=0$ form a set of three equations with three unknowns $z,\rho,\nu$, which can be solved generically. So for $N=1$, we can indeed find generic eigenvectors which are exponentials. 

Also note that inserting $q(x,y) = \cosh(\rho x) \cosh(\rho y)$ into the right-hand side of \eqref{eq:equiv_N1} gives exactly $P_z(\rho,\nu)=0$ for the PDE and $k(z) \cosh(\rho a) + \rho \sinh(\rho a) = k(z) \cosh(\nu b) + \rho \sinh(\nu b) =0$ for the boundary condition.

We claim that with this theorem we have characterized all the eigenvalues. We will prove this by showing that we can construct a resolvent for all other values $z$. We do this by first constructing a basis of the eigenfunctions of the operator $L(z)$ that satisfy the boundary conditions.

\subsection{Sturm-Liouville problems arising from neural field equations}\label{s:SLP}
Solving the characteristic equation $\Delta(z)q=0$ is equivalent to finding the solution of the boundary value problem \eqref{PDE_N1} and \eqref{BC}. Throughout this section we omit the $z$-dependence of $k$ and $c$ for clarity. We seek solutions of the PDE \eqref{PDE_N1} in the separated variable form
\begin{equation}\label{separate}
q(x,y)=\phi(x)\psi(y), (x,y)\in\Omega.
\end{equation}
Inserting this form into \eqref{PDE_N1}, we obtain
\[
\phi''(x)\psi''(y)-k^2(\phi''(x)\psi(y)+\phi(x)\psi''(y))+\Bigl( k^4-\frac{4ck^2}{z+\alpha}\Bigr)\phi(x)\psi(y)=0.
\]
Dividing by $\phi(x)\psi(y)$ and denoting $\phi''/\phi=f(x)$ and $\psi''/\psi=g(y)$,
\[
f(x)\left(g(y)-k^2\right)=k^2g(y)-k^4+\frac{4ck^2}{z+\alpha},
\]
or equivalently,
\[
f(x)=k^2+\frac{4ck^2}{(z+\alpha)(g(y)-k^2)}.
\]
Letting $\rho^2\in\mathbb{C}$ be the common constant value, we obtain two Sturm-Liouville problems (SLP) with Robin type 
boundary conditions
\begin{equation}\label{SLx}
\begin{cases}
\phi''(x)-\rho^2\phi(x)=0,\quad x\in[-a,a]\\
k\phi(-a)-\phi'(-a)=0 \\
k\phi(a)+\phi'(a)=0 
\end{cases}
\end{equation}  
and
\begin{equation}\label{SLy}
\begin{cases}
\psi''(y)-\nu^2\psi(y)=0,\quad y\in[-b,b]\\
k\psi(-b)-\psi'(-b)=0 \\
k\psi(b)+\psi'(b)=0, 
\end{cases}
\end{equation}  
where $\nu^2=k^2-\frac{4ck^2}{(z+\alpha)(k^2-\rho^2)}$. Note that the boundary conditions follow from Lemma~\ref{normalBC}.

Let us introduce the coordinate transformation $\tilde x =\frac{\pi}{2a}x$ into the SLP \eqref{SLx} and obtain
\begin{equation}\label{SLxtildeBC}
\left(\frac{\pi}{2a}\right)^2\phi''(\tilde x)-\rho^2\phi(\tilde x)=0,\quad \tilde x\in\left[-\pi/2,\pi/2\right].
\end{equation}  
Equivalently, the SLP is 
\begin{equation}\label{SLxtilde}
\begin{cases}
\phi''(\tilde x)+\lambda\phi(\tilde x)=0, \quad \tilde x\in\left[-\frac{\pi}{2},\frac{\pi}{2}\right],\ \lambda =-\left(\frac{2a}{\pi}\rho\right)^2,\\[10pt]
\Gamma_1(\phi):=k\phi\left(-\pi/2\right)-\frac{\pi}{2a}\phi'\left(-\pi/2\right)=0, \\[10pt]
\Gamma_2(\phi):=k\phi\left(\pi/2\right)+\frac{\pi}{2a}\phi'\left(\pi/2\right)=0. 
\end{cases}
\end{equation}

First, we check separately the case when $\lambda=0$ is an eigenvalue. Here there are two cases, either $k=k(z)=0$, then the problem reduces to the homogeneous Neumann boundary conditions and the eigenfunction corresponding to the zero eigenvalue is $\phi(\tilde x)=1$, or $k(z)=-\tfrac{1}{a}$, then $\phi(\tilde x)= \tilde x$ is a solution. 

We study the SLP \eqref{SLxtilde} by first rewriting it to a first order system as
\[
Y'(\tilde x)=(P-\lambda W)Y(\tilde x),\quad 
Y=\begin{pmatrix}
\phi\\
\phi'
\end{pmatrix},\quad \tilde x\in[-\pi/2,\pi/2],
\]
with 
\[
P=
\begin{pmatrix}
0&1\\
0&0
\end{pmatrix}, \quad 
W=
\begin{pmatrix}
0&0\\
1&0
\end{pmatrix}. 
\]
The boundary conditions can be reformulated as
\begin{equation}
A Y(-\pi/2)+BY(\pi/2)=0, \quad \text{with }
A=\begin{pmatrix} k&-\pi/2a\\ 0&0 \end{pmatrix},\quad  B=\begin{pmatrix}  0&0\\k&\pi/2a\end{pmatrix}.
\end{equation}
Let $\Phi(\cdot\,;x_0,\lambda)$ be the matrix solution of the initial value problem
\[
\Phi'=(P-\lambda W)\Phi, \quad \Phi(x_0;x_0,\lambda)=I,\ x_0\in[-\pi/2,\pi/2],\ \lambda\in\mathbb{C},
\]
with $I$ the identity matrix, and define the following transcendental function, also called characteristic function for the SLP 
\begin{equation}\label{char-functionSL}
\chi(\lambda)=\det\left(A+B\Phi(\pi/2;-\pi/2,\lambda)\right),\ \lambda\in\mathbb{C}.
\end{equation}
Let us recall the following lemma that shows that the zeroes of $\chi$ are precisely the eigenvalues of the SLP.

\begin{lem}[Lemma 3.2.2, \cite{SLproblems}] \label{lemma-eigenvaluesSL}
A complex number $\lambda$ is an eigenvalue of the BVP \eqref{SLxtilde} if and only if $\chi(\lambda)=0$. Furthermore, 
the geometric multiplicity of the eigenvalue $\lambda$ is equal to the number of linearly independent vector solutions $C=Y(-\pi/2)$ of the linear algebra system 
\[ \left[A+B\Phi(\pi/2;-\pi/2,\lambda)\right]C=0.
\]
\end{lem}
To compute $\Phi$ choose the branch $\mu=\sqrt{\lambda},$ $\lambda\not= 0$, and obtain
\begin{equation}
\Phi(\tilde x; x_0, \lambda)=\begin{pmatrix}
\cos(\mu(\tilde x-x_0))&\frac{1}{\mu}\sin(\mu(\tilde x-x_0))\\
-\mu\sin(\mu(\tilde x-x_0))&\cos(\mu(\tilde x-x_0))\end{pmatrix}.
\end{equation} 
Hence the characteristic function is 
\begin{align}\label{eigenvaluesSLxtilde}
\chi(\mu)&= k\frac{\pi}{a}\cos(\pi\mu)+\left(k^2\frac{1}{\mu}-\left(\frac{\pi}{2a}\right)^2  \mu\right)\sin(\pi \mu)=0,\ \mu=\sqrt{\lambda}.
\end{align}
Since $\chi(-\mu)=\chi(\mu)$, if $\mu\in\mathbb{C}$ is a root of the entire function $\chi$, then so is $-\mu$. According to Lemma~\ref{lemma-eigenvaluesSL}, the eigenvalues of the SLP \eqref{SLxtilde} are exactly the roots of the equation \eqref{eigenvaluesSLxtilde}. Note that this equation has infinite but countable number of roots, these are all simple and have no finite accumulation point in $\mathbb{C}$. A few roots of the SLP \eqref{SLx} are plotted in Figure~\ref{eigenvaluesSL} for $k=1.4-1.4i$ and $a=\pi$.
\begin{figure}
    \centering  
    \includegraphics[scale=.3]{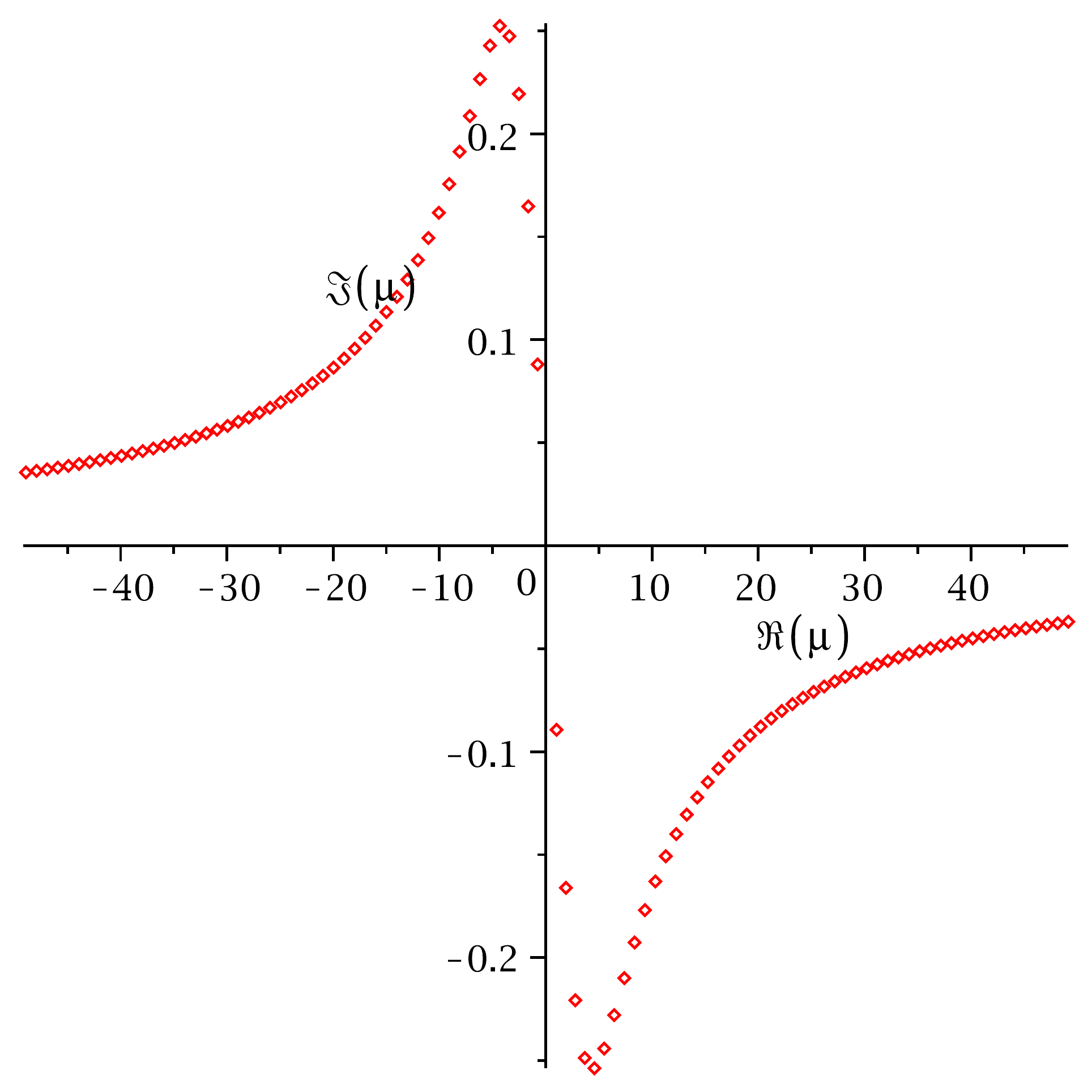}
    \caption{Some roots of the characteristic function $\chi(\mu)$ in \eqref{eigenvaluesSLxtilde}, when $k=1.4-1.4i$ and $a=\pi$.}\label{eigenvaluesSL}
\end{figure}
\subsubsection{Eigenvalues and completeness of exponentials}\label{s:completeness_exp}
In this section we describe the location of the eigenvalues of the SLP in the complex plane and an interesting consequence which results from this distribution of the eigenvalues, i.e., the sets of exponential functions $\{e^{\pm \rho_n x}\}$ and $\{e^{\pm \nu_n y}\}$ used to construct the solutions of the corresponding Sturm-Liouville problems \eqref{SLx} and \eqref{SLy} are complete in $C([-a,a])$ and $C([-b,b])$, respectively. There is an extensive literature on the completeness of sets of complex exponential functions over finite intervals, see e.g., \cite{completeness_of_exp,completeness_of_basis_exp,young} and the references therein. In Section~\ref{s:completenes_eigSLP}, we will construct the eigenfunctions corresponding to the eigenvalues of the SLP using these exponentials. Note here that although we were not successful in proving the completeness of the eigenfunctions in the corresponding Banach space of continuous functions, but only in the larger space of square integrable functions, we were able to overcome this in Section~\ref{s:resolvent} and give a complete characterization of the spectrum of the DDE. 

Let us introduce some notations. The set of zeroes of a function $f$ is denoted by
\begin{equation}
  \mathcal{Z}(f)=\{\mu \mid f(\mu)=0\}
\end{equation}
and the number of zeroes of $f$ by $\nr\mathcal{Z}(f)$. In general, the cardinality of a set will be denoted by $\nr(\cdot)$.

Using the coordinate transformation introduced in \eqref{SLxtildeBC}, we will show that the set of complex exponentials $\{e^{i\mu_n \tilde x}\mid \mu_n\in\mathcal{Z}(\chi)\}$ is complete in $C\left(\left[-\pi/2,\pi/2\right]\right)$, where $\chi$ is the characteristic function \eqref{eigenvaluesSLxtilde}. Since $\chi(-\mu)=\chi(\mu)$, if we denote the roots of $\chi$ for which $\real\mu_n>0$ by $\mu_n$, then $\mu_{-n}=-\mu_n$ are also roots of $\chi$. This sequence is called symmetric and we can denote it by $\{\mu_n\}_{n=-\infty}^\infty$.

In the following theorem, we summarize two important results from \cite{young} that we will use to prove the completeness of sets of exponentials.
\begin{thm}\label{th:Young}
Let $\{\mu_n\}_{n=-\infty}^\infty$ be a sequence of complex numbers. 
\begin{itemize}
    \item[(1)] If 
    \begin{equation}\label{conditions completeness}
    \sup_n|\real\mu_n -n|<\frac{1}{4}\text{ and } \sup_n|\imag\mu_n|<\infty,
    \end{equation}
    then the system $\{e^{i\mu_n x}\}_{n=-\infty}^\infty$ is complete in $C(I)$, for each closed subinterval $I$ of $(-\pi,\pi)$.
    \item[(2)] The completeness of the system $\{e^{i\mu_n x}\}$ in $C(I)$ is unaffected if some $\mu_n$ is replaced by an other (different from all) number (\cite{young}, Theorem 7, Chapter 3).
\end{itemize}
\end{thm}
The conditions in \eqref{conditions completeness} show that all $\mu_n$ lie "near" the real axis. Our next two lemmas show that almost all roots of $\chi$ satisfy these conditions.
\begin{lem}\label{lem:real part}
    Consider the characteristic equation $\chi(\mu)=0$ in \eqref{eigenvaluesSLxtilde}. For every $k\in\mathbb{C}$ there exists $N\in\mathbb{N}$ and $0<\epsilon<1$ such that the followings hold
    \begin{itemize}
        \item[(a)] for all $n\geq N$ there is a unique $\mu_n\in\mathcal{Z}(\chi)$, such that $|\mu_n-n|<\epsilon$.
        \item[(b)] $\nr\{\mu\in\mathcal{Z}(\chi) : |\real \mu|\leq N\}=2N+2.$
    \end{itemize}
\end{lem}
\begin{proof}
We prove part (a) first. Let us rewrite $\chi(\mu)=0$ in \eqref{eigenvaluesSLxtilde} as
\begin{equation}\label{Fmu}
    F(\mu)=\tan(\pi\mu)-\frac{2 h \mu}{\mu^2-h^2}=0,\quad h=\frac{2a}{\pi}k.
\end{equation}
Note that if $F(\mu)=0$ then $F(-\mu)=0$. Since $\tan(\pi\mu)$ has zeroes at $n$ and poles at $1/2+n$, where $n\in\mathbb{Z}$, consider the following stripe in the complex plane
\[
S=\left\{z\in\mathbb{C}\mid \real z\in\left(-\frac{1}{4},\frac{1}{4}\right)\right\}.
\]
Let $\Gamma$ be a closed simple curve around $z=0$, with interior $G$, that is contained entirely in $S$ and avoids $\pm h$. The complement domain of $G$ in $S$ is denoted by $G^c$. Use the following two observations. First, 
\begin{equation}\label{min_lhs}
    \min_{\mu\in G^c}|\tan(\pi\mu)|=\epsilon>0,
\end{equation}
and second, that 
\begin{equation}\label{lim_rhs}
 \lim_{|\mu|\to\infty}\frac{2 h \mu}{\mu^2-h^2}=0\ \text{ for all } h\in\mathbb{C}.   
\end{equation}
The limit \eqref{lim_rhs} yields that there exists an $N_1\in\mathbb{N}$, such that
\[
\sup_{|\real\mu|>N_1}\left|\frac{2 h \mu}{\mu^2-h^2}\right|<\epsilon.
\]
Let $\tilde\Gamma$ be the shift of $\Gamma$ by $n\in\mathbb{N}$, where $n\geq N_1$ and $\pm h\not\in \tilde\Gamma$. Since 
\begin{equation}
 |\tan(\pi\mu)|\geq\epsilon,\quad  \left|\frac{2 h \mu}{\mu^2-h^2}\right|<\epsilon\quad \text{for all }\mu\in\tilde\Gamma,
\end{equation}
we can apply Rouch\'e's theorem to conclude that the number of zeros of $F$ in the interior of $\tilde\Gamma$ equals the number of zeros of $\tan(\pi\cdot)$. Due to the choice of $\tilde\Gamma$, we obtained that $F$ has exactly one zero inside $\tilde\Gamma$, which completes the proof of part (a) of the lemma.

To prove part (b), we study the zeroes of $F$, or equivalently, the solutions of 
\begin{equation}  
  \tan(\pi\mu)=\frac{2 h \mu}{\mu^2-h^2}=:f_h(\mu).
\end{equation}
When $h=0$, the roots are $\mu_n =n$, where $n\in\mathbb{Z}$. Note that in this case the SLP reduces to the Neumann problem and we know that the system $\{e^{i n x}\}_{n\in\mathbb{Z}}$ is complete in $C\left(\left[-\pi/2,\pi/2\right]\right)$.

Let us fix $h\not=0$ and let $\mu=x+iy$. Using the trigonometric relation
\[
\left|\tan(\pi(x+iy))\right|^2=\frac{\cosh(2\pi y)-\cos(2\pi x)}{\cosh(2\pi y)+\cos(2\pi x)},
\]
we have that, for $y\not=0$,
\begin{equation} 
\min_{x\in\mathbb{R}}\left|\tan(\pi(x+iy))\right|^2=\frac{\cosh(2\pi y)-1}{\cosh(2\pi y)+1}>0.
\end{equation}
Moreover, if $x=\pm(n+1/4)$, then $\left|\tan(\pi(x+iy))\right|^2=1$ for all $y\in\mathbb{R}$. The limit
\[
\lim_{|\mu|\to\infty}f_h(\mu)=0,\ \forall h\in\mathbb{C}
\]
suggests to take for given $h$ a closed curve around the origin, such that $|f_h(\cdot)|^2$ is small on this curve and it is far from the poles of $f_h$. Hence, define a square around the origin as
\[
\Gamma_{h,n}=\left\{\pm \left(n+\frac{1}{4}\right)+iy, y\in\left[-n-\frac{1}{4},n+\frac{1}{4}\right]\right\}\cup
\left\{x\pm i\left(n+\frac{1}{4}\right), x\in\left[-n-\frac{1}{4},n+\frac{1}{4}\right]\right\}.
\]
Then
\begin{equation}
\lim_{n\to\infty}\max_{\mu\in\Gamma_{h,n}}|f_h(\mu)|^2= 0\text{ and } \lim_{n\to\infty}\min_{\mu\in\Gamma_{h,n}}|\tan(\pi\mu)|^2= 1.
\end{equation}
Therefore, there exists $N_2\in\mathbb{N}$ such that the poles of $f_h$, that is $\pm h$, are contained in the interior of $\Gamma_{h,N_2}$ and $|h-\mu|>1$ for all $\mu\in \Gamma_{h,N_2}$, and $N_2$ is large enough such that 
\begin{equation}
 |f_h(\mu)|^2<1/2 \text{ and } |\tan(\pi\mu)|^2>1/2 \text{ for all }\mu\in\Gamma_{h,N_2}.  
\end{equation}
Then we can apply Rouch\'e's theorem, which says that in the interior of $\Gamma_{h,N_2}$
\[
\nr\mathcal{Z}(\tan(\pi\cdot))-\nr\mathcal{P}(\tan(\pi\cdot))=\nr\mathcal{Z}(F)-\nr\mathcal{P}(F),
\]
where $\nr\mathcal{P(\cdot)}$ counts the number of poles of the corresponding functions. The left hand side equals 1 and on the right hand side  
\begin{equation}
 \nr\mathcal{P}(F)=\nr\mathcal{P}(\tan(\pi\cdot))+\nr\mathcal{P}(f_h)=2N_2+2,    
\end{equation}
since $\pm h$ are in the interior of $\Gamma_{h,N_2}$. Hence, $\nr\mathcal{Z}(F)=2N_2+3$. Since $\chi$ has the same zeroes as $F$, except $\mu=0$, we can conclude that $\nr\mathcal{Z}(\chi)=2N_2+2$ in the interior of $\Gamma_{h,N_2}$, which completes part (b) of the lemma. 

Therefore, letting $N=\max\{N_1,N_2\}$ is then suitable for both parts of the lemma.
\end{proof}

\begin{cor}
An immediate consequence of Lemma~\ref{lem:real part} is that 
\[\sum_{\mu\in\mathcal{Z}}\frac{1}{|\mu|}=\infty. \]
\end{cor}
\begin{lem}\label{lemma_arg}
Consider the set $\mathcal{Z}(\chi)=\{\mu_n\}$, with $\chi$ given in \eqref{eigenvaluesSLxtilde}. Then
\begin{equation}\label{argument}
  \sup_n|\imag \mu_n|<\infty. 
\end{equation}
\end{lem}
\begin{proof}
First, note that the set $\mathcal Z(\chi)$ cannot have finite accumulation points. If all  eigenvalues  $\mu_n$,  are real, then \eqref{argument} holds. 

If the assertion is not true, then there exists a subsequence  $\{\mu_{n_k}\}$, such that $|y_{n_k}|\to\infty$, where $\mu_{n_k}=x_{n_k}+iy_{n_k}$. If we insert this into \eqref{Fmu}, we obtain
\[
\tan(\pi(x_{n_k}+iy_{n_k}))=\frac{\tan(\pi x_{n_k})+i\tanh{(\pi y_{n_k})}}{1-i\tan(\pi x_{n_k})\tanh(\pi y_{n_k})}=\frac{2h(x_{n_k}+i y_{n_k})}{(x_{n_k}+i y_{n_k})^2-h^2}.
\]
Taking the limit $|y_{n_k}|\to\infty$, we obtain that the left hand side converges to $i$ and the right hand side to $0$, which leads to a contradiction.
\end{proof}

The main result of this section is the following theorem.
\begin{thm}[Completeness theorem]\label{th:main completeness result}
Let $\mathcal{Z}(\chi)=\{\mu_n\}_{n=-\infty}^{\infty}$, where $\chi$ is the characteristic function in \eqref{eigenvaluesSLxtilde}. Then the set $\{e^{i\mu_n \tilde x}\}_{n=-\infty}^\infty$ is complete in $C([-\pi/2,\pi/2])$.
\end{thm}
\begin{proof}
From Lemma~\ref{lem:real part} it follows that there exists an $N\in\mathbb{N}$ such that 
\begin{equation}
    \sup_{|n|>N}|\real\mu_n -n|<\frac{1}{4}.
\end{equation}
Moreover, since $\nr\{\mu\in\mathcal{Z}(\chi) : |\real \mu|\leq N\}=2N+2$, let us replace $2N+1$ of these roots by $n$ in the exponentials, that is by $\{e^{in \tilde x}\}_{n=-N}^N$. The new set of exponentials will now satisfy the condition on the real part of the eigenvalues in \eqref{conditions completeness}. This, in combination with \eqref{argument}, implies that this set is complete in each closed subinterval of $(-\pi,\pi)$. According to part (2) of Theorem~\ref{th:Young}, if we replace the set $\{e^{in \tilde x}\}_{n=-N}^N$ by the corresponding finite set $\{e^{i\mu_n \tilde x}\}$, then the completeness will be unaffected. 
\end{proof}

This way we have shown that the set $\{e^{\pm \rho_n x}\}$ is complete in $C([-a,a])$ and analogously, we can show that the set $\{e^{\pm \nu_m y}\}$ is complete in $C([-b,b])$. Then $\{e^{\pm \rho_n x}e^{\pm \nu_m y}\}$ forms a complete set in $C(\bar\Omega)$. 
\subsubsection{Completeness of the eigenfunctions}\label{s:completenes_eigSLP}
In  \cite{SLP_Marchenko}, the solution of SLP problems are discussed in a more general framework. Based on these results, we construct the eigenfunctions corresponding to the eigenvalues of the SLP and state their completeness in the space of square integrable functions. As a consequence, the eigenfunctions of the the operator $L(z)$, which are the separable solutions of the BVP \eqref{eq:equiv_N1}, form a complete basis in $L^2(\Omega)$. In Section~\ref{s:resolvent} we show that this is sufficient to give a complete characterization of the spectrum of the DDE and to solve the resolvent problem in Section~\ref{s:resolvent}. 

Using the results of the previous section, we can conclude that the eigenvalues of the SLP \eqref{SLxtilde} with large absolute value are of the form
\begin{equation}\label{eigenvalue_form}
    \lambda_n=\left(n+\frac{\delta_n}{n}\right)^2,\quad \text{where } \sup|\delta_n|<\infty.
\end{equation}
Note that this was also shown in \cite{SLP_Marchenko}. Following the ideas and results in \cite{SLP_Marchenko}, we can construct the following solutions of the differential equation in \eqref{SLxtilde}
\begin{align}
    \phi_1(\mu,\tilde x)&=-\frac{\pi}{2a}\cos\left(\mu(\tilde x +\frac{\pi}{2})\right)-k\frac{1}{\mu}\sin\left(\mu(\tilde x +\frac{\pi}{2})\right) \\
    \phi_2(\mu,\tilde x)&=\frac{\pi}{2a}\cos\left(\mu(\tilde x +\frac{\pi}{2})\right)-k\frac{1}{\mu}\sin\left(\mu(\tilde x -\frac{\pi}{2})\right). 
\end{align}
Moreover, since 
\[
\Gamma_1(\phi_1)=\Gamma_2(\phi_2)=0,\quad
\Gamma_1(\phi_2)=-\Gamma_2(\phi_1)=\chi(\mu),
\]
if $\mu$ is a root of the characteristic function \eqref{eigenvaluesSLxtilde}, then $\phi_1$ and $\phi_2$ solve the SLP \eqref{SLxtilde} and they are called eigenfunctions corresponding to the eigenvalue $\lambda$, with $\mu=\sqrt\lambda$. 

\begin{thm}[Theorem 1.3.2., \cite{SLP_Marchenko}]\label{thm:completenessL2}
The system of eigenfunctions and generalized eigenfunctions of the BVP \eqref{SLxtilde} is complete in the space $L^2\left((-\pi/2,\pi/2)\right)$ and constitutes there a Riesz basis.
\end{thm}
Since all roots of the characteristic function $\chi(\mu)$ are simple, we do not have generalized eigenfunctions for this problem. The linear span of the eigenfunctions constructed from the solutions $\phi_1$ and $\phi_2$ coincide. It is therefore sufficient to consider the eigenfunctions derived from $\phi_1$ hence, we will omit the subscript. 

If $\mathcal{Z}(\chi)=\{\mu_n\}$, with $\mu_n=\sqrt{\lambda_n}$, then the corresponding eigenfunctions will be denoted as
\begin{equation}
\begin{split}
    \phi_n(\tilde x):=& \phi(\mu_n,\tilde x)=-\frac{\pi}{2a}\cos\left(\mu_n(\tilde x +\frac{\pi}{2})\right)-k\frac{1}{\mu_n}\sin\left(\mu_n(\tilde x +\frac{\pi}{2})\right)\\[5pt]
    =&\tilde f(\mu_n,a)\cos(\mu_n\tilde x)+\tilde g(\mu_n,a)\sin(\mu_n\tilde x),\quad \tilde x\in [-\pi/2,\pi/2],
\end{split}
\end{equation}
where  
\[
    \tilde f(\mu,a)=-\frac{\pi}{2a}\cos(\mu\frac{\pi}{2})-\frac{k}{\mu}\sin(\mu \frac{\pi}{2}),\quad 
    \tilde g(\mu,a)=\frac{\pi}{2a}\sin(\mu\frac{\pi}{2})-\frac{k}{\mu}\cos(\mu \frac{\pi}{2})
\]
and the following identity holds
\[
    2\mu \tilde f(\mu,a)\tilde g(\mu,a)=\chi(\mu).
\]
From here, it follows that if $\mu\not=0$ is a root of the characteristic function, then either $\tilde f(\mu,a)=0$ or $\tilde g(\mu,a)=0$. From \eqref{eigenvalue_form} it follows that the roots of $\tilde f$ are those roots of $\chi$ that have the form $\mu_{2n-1}=2n-1+\frac{\delta_{2n-1}}{2n-1}$ and the corresponding eigenfunctions we call \emph{odd eigenfunctions} and they have the form 
\begin{equation}
    \phi_{2n-1}(\tilde x)=\tilde g(\mu_{2n-1},a)\sin(\mu_{2n-1}\tilde x),\ n=1,2,\dots.
\end{equation}
Similarly, the roots of $\tilde g$ have the form $\mu_{2n}=2n+\frac{\delta_{2n}}{2n}$ and the corresponding eigenfunctions are called \emph{even eigenfunctions}   
\begin{equation}
    \phi_{2n}(\tilde x)=\tilde f(\mu_{2n},a)\cos(\mu_{2n}\tilde x),\ n=0,1,2,\dots.
\end{equation}
Summarizing, Theorem~\ref{thm:completenessL2} implies that the set of even and odd eigenfunctions $\{\phi_{2n-1}(\tilde x),\phi_{2n}(\tilde x)\}_{n=1}^\infty$ is complete in $L^2((-\pi/2,\pi/2))$.

In the original coordinate system the eigenfunctions are 
\begin{equation}
    \begin{split}
   \phi_n(x)&= -\frac{\pi}{2a}\cosh\left(\rho_n( x +a)\right)-k\frac{\pi}{2a}\frac{1}{\rho_n}\sinh\left(\rho_n( x +a)\right)\\
    &=-\frac{\pi}{2a}\left(f(\rho_n, a)\cosh(\rho_n x)+g(\rho_n,a)\sinh(\rho_n x)\right),\quad x\in[-a,a], 
\end{split}
\end{equation}
where 
\[
    f(\rho,a)=\cosh(\rho a)+\frac{k}{\rho}\sinh(\rho a),\quad 
    g(\rho,a)=\sinh(\rho a)+\frac{k}{\rho}\cosh(\rho a)
\]
and the following holds
\[
    2\rho f(\rho,a)g(\rho,a)=\frac{2a}{\pi}\chi(\rho)=2k\cosh(2a\rho)+\left(\frac{k^2}{\rho}+\rho\right)\sinh(2a\rho).
\]
Using the same argument as before, if $\rho$ is a root of $\chi(\rho)$, then either $f$ or $g$ vanish there. Note that, these are precisely the conditions we obtained earlier in Theorem~\ref{thm:eigenvectors}. We can conclude that the system $\{\phi_{2n-1}(x),\phi_{2n}(x)\}_{n=1}^\infty$ is complete in $L^2([-a,a])$, with
\begin{equation}
    \phi_{2n-1}(x)=g(\rho_{2n-1},a)\sinh(\rho_{2n-1} x),\quad \phi_{2n}(x)=f(\rho_{2n},a)\cosh(\rho_{2n} x).
\end{equation}
Note that the following relations hold
\[
\tilde f(\mu_{2n},a)=-\frac{\pi}{2a}f(\rho_{2n},a),\quad \tilde g(\mu_{2n+1},a)=i\frac{\pi}{2a}g(\rho_{2n+1},a).
\]

Analogous result holds for the eigenvalues and corresponding eigenfunctions of the SLP \eqref{SLy}. Returning to the original problem of solving the BVP \eqref{PDE_N1} and \eqref{BC} by separating the variables, we can summarize as follows. Consider the bilinear mapping
\begin{align*}
L^2([-a,a])\times L^2([-b,b])&\to L^2\left([-a,a]\times [-b,b]\right)=L^2(\bar\Omega)\\
(\phi,\psi)&\mapsto \phi\psi.
\end{align*}
The set of linear combinations of functions of the form $\phi(x)\psi(y)$ 
is dense in $L^2\left(\bar\Omega\right)$ since $L^2([-a,a])$ and $L^2([-b,b])$ are separable  (it contains a countable, dense subset). Using that the eigenfunctions $\{\phi_m$\} and $\{\psi_n\}$ are complete in $L^2([-a,a])$ and $L^2([-b,b])$, respectively, we can conclude that the product of the eigenfunctions 
$\{\phi_n(x)\psi_m(y)\}$ is complete in $L^2\left(\bar\Omega\right)$. Note that if $\rho$ and $\nu$ are the eigenvalues of the SLP \eqref{SLx} and \eqref{SLy}, respectively, then the corresponding boundary conditions in the SLP are precisely the conditions in Theorem~\ref{thm:eigenvectors}.
Consequently, $\{\phi_m\psi_n\}$ give a unique basis expansion in $L^2(\bar\Omega)$, where $\phi_m \psi_n$ satisfies the boundary condition 
\[(K(z)L(z)-L(z)K(z))\phi_m \psi_n = 0\] 
and $L(z)\phi_m \psi_n = (k^2(z)-\rho_m^2)(k^2(z)-\nu_n^2)\phi_m \psi_n$, where $k^2(z)\neq \rho_m^2$ and $k^2(z)\neq \nu_n^2$. 

\subsection{Characterisation of the spectrum and resolvent set of the DDE}\label{s:resolvent}
We are now able to fully characterize the spectrum and resolvent sets of our neural field model for $N=1$. 

\begin{thm}\label{thm:spectrum-resolvent}
Let $z\in \mathbb{C}\setminus\{-\alpha \}$, such that $k(z)\neq 0$. Moreover, let $\{\phi_m \psi_n\}_{ m,n\in \mathbb{N}}$ form a basis of $L^2(\Omega)$, such that $L(z)\phi_m \psi_n = (k^2(z)-\rho_m^2)(k^2(z)-\nu_n^2)\phi_m \psi_n$, where $k^2(z)\neq \rho_m^2$ and $k^2(z)\neq \nu_n^2$, and $(K(z)L(z)-L(z)K(z))\phi_m \psi_n = 0$. 

If there exist $m,n\in \mathbb{N}$ for which $P_z(\rho_m,\nu_n)=0$, then $\Delta(z) \phi_m\psi_n=0$ and $z\in \sigma_p(A)$ with eigenvector $\phi_m\psi_n$.

Otherwise, when $P_z(\rho_m,\nu_n)\neq 0$ for all $m,n\in \mathbb{N}$, then $z\in \rho(A)$. 
\end{thm}
\begin{proof}
Let $z\in \mathbb{C}\setminus\{-\alpha \}$, such that $k(z)\neq 0$ and let $\phi_m \psi_n$ as in the theorem statement. Suppose there exist $m,n\in \mathbb{N}$ for which $P_z(\rho_m,\nu_n)=0$. Then
\[L(z)\Delta(z)\phi_m \psi_n = (z+\alpha)L(z)\phi_m \psi_n - 4c(z)k^2(z) \phi_m \psi_n = P_z(\rho_m,\nu_n) \phi_m \psi_n = 0.\]
Hence by Theorem \ref{thm:equivalence}, $\Delta(z)\phi_m\psi_n=0$.

On the other hand suppose now that $P_z(\rho_m,\nu_n)\neq 0$ for all $m,n\in \mathbb{N}$. In order to prove that $z\in \rho(A)$ it is sufficient to show that $\Delta(z) q=0$ has a unique solution $q\equiv 0$, as $z \notin \sigma_{ess}(A)= \{-\alpha \}$. 

Let $q \in Y$ such that $\Delta(z) q = 0$. As $\Omega$ is a bounded domain we have that $q\in L^2(\Omega)$ and hence it has a unique basis expansion
\begin{equation}\label{sum_q}
q(x,y) = \sum_{m,n} \xi_{m,n} \phi_m(x) \psi_n(y).
\end{equation}
For the following argument we consider $\Delta(z)$ to be an operator from $L^2(\Omega)$ to $L^2(\Omega)$. In this sense it has a bounded operator norm, as the kernel $J$ is $L^2$-integrable. Therefore, we can interchange $\Delta(z)$ with the infinite sum. Using the properties of $\phi_m \psi_n$ we obtain that
\begin{align*}
\Delta(z) \phi_m \psi_n &= (z+\alpha)\phi_m \psi_n - K(z)\phi_m \psi_n\\
&= (z+\alpha)\phi_m \psi_n - \frac{K(z)L(z)\phi_m \psi_n}{(k^2(z)-\rho_m^2)(k^2(z)-\nu_n^2)} \\
&= (z+\alpha)\phi_m \psi_n - \frac{L(z)K(z)\phi_m \psi_n}{(k^2(z)-\rho_m^2)(k^2(z)-\nu_n^2)}\\
&= (z+\alpha)\phi_m \psi_n - \frac{4 c(z) k^2(z)}{(k^2(z)-\rho_m^2)(k^2(z)-\nu_n^2)}\phi_m \psi_n\\
&= Q_z(\rho_m,\nu_n)\phi_m \psi_n.
\end{align*}
Combining this with the sum \eqref{sum_q}, gives
\[\Delta(z) q = \sum_{m,n} \xi_{m,n} \Delta(z)\phi_m \psi_n = \sum_{m,n} \xi_{m,n} Q_z(\rho_m,\nu_n)\phi_m \psi_n = 0.\]
From here we can conclude that $\xi_{m,n}Q(\rho_m,\nu_n)=0$ for all $m,n\in \mathbb{N}$ and
\[Q_z(\rho_m,\nu_n) = \frac{P_z(\rho_m,\nu_n)}{(k^2(z)-\rho_m^2)(k^2(z)-\nu_n^2)}\neq 0.\]
Hence $\xi_{m,n}=0$ for all $m,n\in \mathbb{N}$ and therefore, $q(x,y)=0$  $\forall (x,y)\in\bar\Omega$. 
\end{proof} 
Note that the eigenvectors found are exactly those in Theorem \ref{thm:eigenvectors}.

For $z\in \rho(A)$ we can construct a solution for the resolvent problem, which we need in the next section. 
\begin{prop}
Let $z\in \rho(A)$ such that $k(z)\neq 0$ and let $g\in Y$.
There exists a unique $q\in Y$ that solves 
\begin{equation}\label{resolvent_eq}
  \Delta(z) q = g,  
\end{equation}
and is given by
\begin{equation}\label{eq:resolventsol}
q(x,y)=\frac{g(x,y)}{z+\alpha} + \frac{4 c(z)k^2(z)}{z+\alpha}\sum_{m,n} \frac{\xi_{n,m}}{P_z(\rho_n,\nu_m)} \phi_n(x) \psi_m(y),
\end{equation}
where $\phi_n \psi_m$ are as in Theorem \ref{thm:spectrum-resolvent}.
\end{prop}
\begin{proof}
Let $z\in \rho(A)$ such that $k(z)\neq 0$ and let $g\in Y$. Furthermore let $\{\phi_m \psi_n\}_{m,n\in \mathbb{N}}$ form a basis of $L^2(\Omega)$ such that $L(z)\phi_m \psi_n = (k^2(z)-\rho_m^2)(k^2(z)-\nu_n^2)\phi_m \psi_n$, where $k^2(z)\neq \rho_m^2$ and $k^2(z)\neq \nu_n^2$, and $(K(z)L(z)-L(z)K(z))\phi_m \psi_n = 0$. 

First let us rewrite $q$ as
\[
q(x,y) = \frac{p(x,y)+g(x,y)}{z+\alpha}.
\]
Then $\Delta(z) q = g$ is equivalent to
\begin{equation}\label{eq:resolvent_p}
\Delta(z) p = K(z)g. 
\end{equation}
This implies that $p =  \tfrac{1}{z+\alpha}K(z)(p +g)$. Hence $p$ is in the range of $K(z)$, which implies that it satisfies the smoothness conditions of Theorem \ref{thm:equivalence}. By this theorem, \eqref{eq:resolvent_p} is equivalent to
\begin{equation}
\Delta(z) p = K(z)g \Leftrightarrow \left\{ L(z)\Delta(z)p= 4 c(z)k^2(z) g \text{ and } K(z)L(z)p=L(z)K(z)p\right\}.
\end{equation}

Similar to the previous theorem, we write a unique basis expansion of $g$ as
\begin{equation}
g(x,y) = \sum_{m,n} \xi_{m,n} \phi_m(x) \psi_n(y). 
\end{equation}
By Theorem \ref{thm:spectrum-resolvent} we get that for all $m,n\in \mathbb{N}$, $P_z(\rho_m,\nu_n)\neq 0$. Furthermore, by Lemma~\ref{lem:real part} we have that $|\rho_m|,|\nu_n| \rightarrow \infty$, when $m,n \rightarrow \infty$. Hence $1/|P_z(\rho_n,\nu_m)|\rightarrow 0$ when $m \rightarrow \infty$ or $n \rightarrow \infty$. Then using the properties of $\phi_m\psi_n$ we find that 
\begin{equation}
p(x,y) = 4 c(z)k^2(z) \sum_{m,n} \frac{\xi_{n,m}}{P_z(\rho_n,\nu_m)} \phi_n(x) \psi_m(y)
\end{equation}
solves $L(z)\Delta(z)p = 4 c(z)k^2(z) g$. 

Hence the resolvent becomes
\begin{equation}
(\Delta^{-1}g)(x,y)=\frac{g(x,y)}{z+\alpha} + \frac{4 c(z)k^2(z)}{z+\alpha}\sum_{m,n} \frac{\xi_{n,m}}{P_z(\rho_n,\nu_m)} \phi_n(x) \psi_m(y).
\end{equation}
\end{proof}

\section{An example for Hopf bifurcation}\label{s:Hopf_bifurcation}
Oscillations are important features of nervous tissue that can be studied with neural field models. Hence, Hopf bifurcations play an important role in the analysis. When the space is one-dimensional, Hopf bifurcations were studied in \cite{vg} and along with other types of bifurcations also in \cite{koen, veltz}. On two-dimensional domains, numerical experiments were conducted in \cite{faugeras,PM}. In this section we study an example of Hopf bifurcation in the two-dimensional case based on our analytical results.

Assume that $N=1$, hence the connectivity function has the form
\begin{equation}\label{J_inhibition}
J(r,r')=\hat c e^{-\xi\|r-r'\|_1} \quad \forall r,r'\in\bar\Omega,
\end{equation}
where $\hat c,\xi\in\mathbb{C}$ such that $J$ is real valued. The firing rate function is given by 
\begin{equation}\label{S}
S(u)=\frac{1}{1+e^{-\gamma u}}-\frac{1}{2}\quad \forall u\in\mathbb{R},
\end{equation}
where $\gamma$ is the steepness of the sigmoidal and the delay function is as in \eqref{delay function}.

Set $\alpha=1,$ $\tau_0=1$ and the parameters in \eqref{J_inhibition} and \eqref{S} as $\xi=2$ and $\gamma=4$, respectively. The bifurcation parameter in this example is $\hat c$. When $\hat c<0$, this type of connectivity models a population with inhibitory neurons. There is a Hopf bifurcation at $\hat c=-3.27$ with eigenvalues $\lambda = \pm 1.34 i$, see Figure~\ref{fig:spectrum_Hopf}, and corresponding eigenvector
\[\varphi_\lambda(t)(r)= e^{1.34 i t}\cosh((-0.17+1.15 i)x)\cosh((-0.17+1.15 i)y),\ t\in[-\tau_{max},0], r = (x,y)\in\bar\Omega.\]

\begin{figure}
\begin{center}
\includegraphics[scale=.7]{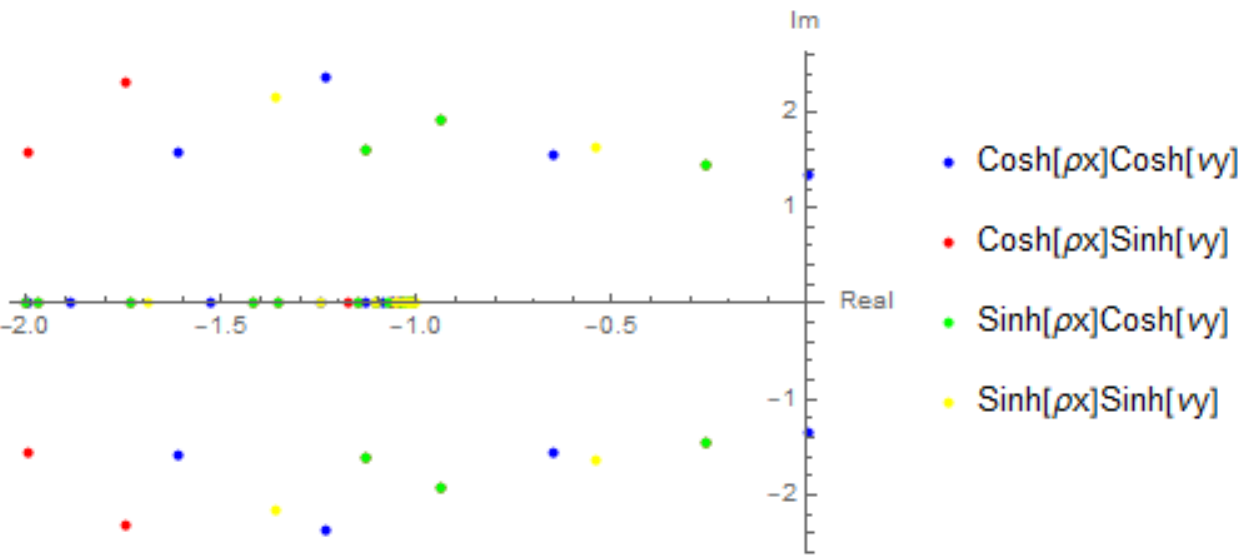}
\caption{Spectrum of the linearized system at a Hopf bifurcation for $\hat c=-3.27, \alpha=1, \tau_0=1, \xi=2, \gamma=4, a=b=1$.}\label{fig:spectrum_Hopf}
\end{center}
\end{figure}

In \cite{vg}, a procedure is derived using the sun-star calculus to compute the Lyapunov coefficient for a Hopf bifurcation. For this we need the higher order Fr\'echet derivatives of $G$ (see \cite{vg}): 
\[D^nG(\hat\varphi)(\varphi_1,\cdots, \varphi_n)(r)=\int_{\Omega}J(r,r')S^{(n)}(\hat\varphi(-\tau(r,r'),r'))\prod_{i=1}^n \varphi_i(-\tau(r,r'),r')dr',\]
for $\varphi_1,\dots,\varphi_n\in X$ and $r\in\bar\Omega$. Due to our choice of $S$, $S''(0)=0$ and therefore, $D^2G(0)$ vanishes. This reduces the computation of the Lyapunov coefficient to the following equality 
\begin{equation}\label{lyapunov_computation}
\frac{1}{2\pi i}\oint_{\partial C_{\lambda}} \Delta^{-1}(z)D^3G(0)(\varphi_\lambda,\varphi_\lambda,\bar{\varphi}_\lambda) dz = g_{21} \varphi_\lambda(0),
\end{equation}
where $\partial C_{\lambda}$ is a closed contour containing $\lambda$ and no other eigenvalues. The first Lyapunov coefficient is given by (see \cite{kuznetsov_2013})
\[l_1 = \frac{\text{Re } g_{21}}{\text{Im } \lambda}.\]

We use \eqref{lyapunov_computation} as an identity for $g_{21}$. To compute the contour integral in \eqref{lyapunov_computation} we take for $C_\lambda$ a small circle of radius $\epsilon$ around $\lambda$, $z=\lambda + \epsilon e^{2\pi i \theta}$ for $0\leq\theta<1$ and perform a change of variables to obtain
\[\int_0^1 \epsilon e^{2\pi \theta} \Delta^{-1}(\lambda + \epsilon e^{2\pi \theta})D^3G(0)(\varphi_\lambda,\varphi_\lambda,\bar{\varphi}_\lambda) d\theta = g_{21} \varphi_\lambda(0).\]
We then compute the integral numerically by using an equidistant grid on $[0,1)$ of $n_z$ points, where we use that $z$ is periodic in $\theta$. 

To compute the resolvent we approximate \eqref{eq:resolventsol} by truncating the infinite sum. For each grid point $\theta$ of above, we compute $n_x$ basis functions $\phi_m(x)$ and $n_y$ basis functions $\psi_n(y)$. So we end up with $n_x n_y$ basis functions $\phi_m(x)\psi_n(y)$. We use the Gramm-Schmidt procedure to get an orthonormal set with respect to the $L^2$ inner product. This enables us to find the coefficients $\xi_{m,n}$ by standard orthogonal projection. We find that using $n_x=n_y=3$ gives a good enough approximation for the purposes of this calculation, especially away from the boundary. 

Finally, we need to compute the scalar $g_{21}$ and find that 
\begin{equation}\label{eq:g21}
g_{21}=\frac{1}{\varphi_\lambda(0)}\int_0^1 \epsilon e^{2\pi \theta} \Delta^{-1}(\lambda + \epsilon e^{2\pi \theta})D^3G(0)(\varphi_\lambda,\varphi_\lambda,\bar{\varphi}_\lambda) d\theta.   
\end{equation}
This right hand side is, however, still a function of $x,y$ instead of a scalar, so naturally, this should be a constant function. We can use this fact to check our calculations. Using the values for the Hopf bifurcation above and $\epsilon = 0.01, n_z=32$, we see in Figure \ref{fig:my_label} that this is indeed the case. This results in a Lyapunov coefficient of $l_1=-1.572$. The negative sign of $l_1$ indicates a supercritical Hopf bifurcation.

\begin{figure}
    \centering
    \includegraphics[scale=.7]{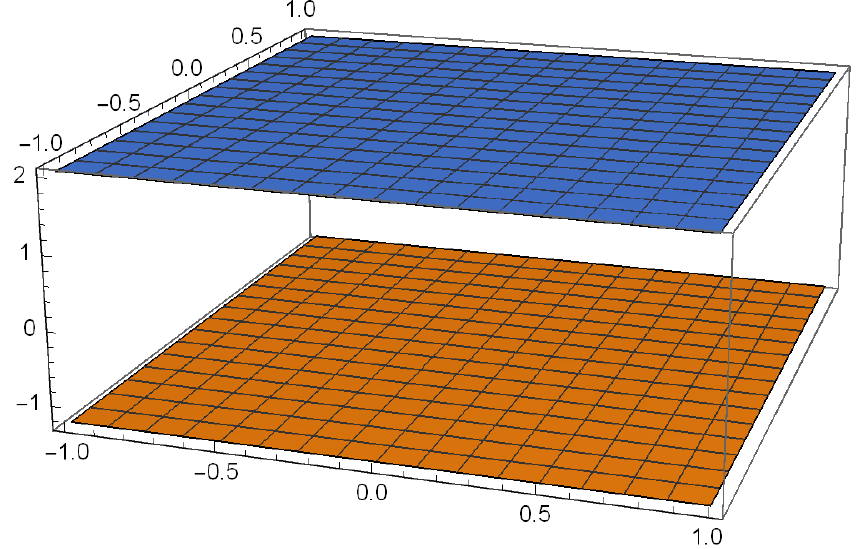}
    \caption{The real (orange) and imaginary part (blue) of $g_{21}$ at the Hopf bifurcation for $\hat c=-3.27$ at $\lambda = \pm 1.34 i$ and with $n_x=n_y=3$, $\epsilon = 0.01, n_z=32$.}
    \label{fig:my_label}
\end{figure}

Some numerical time simulations were performed in Figure~\ref{fig:time_simulation_around_Hopf} and Figure~\ref{fig:time_simulation_beyond_Hopf} to illustrate the dynamic behavior of the solution of the neural filed model for parameter values before and beyond Hopf bifurcation.

\begin{figure}
\begin{center}
\includegraphics[scale=.43]{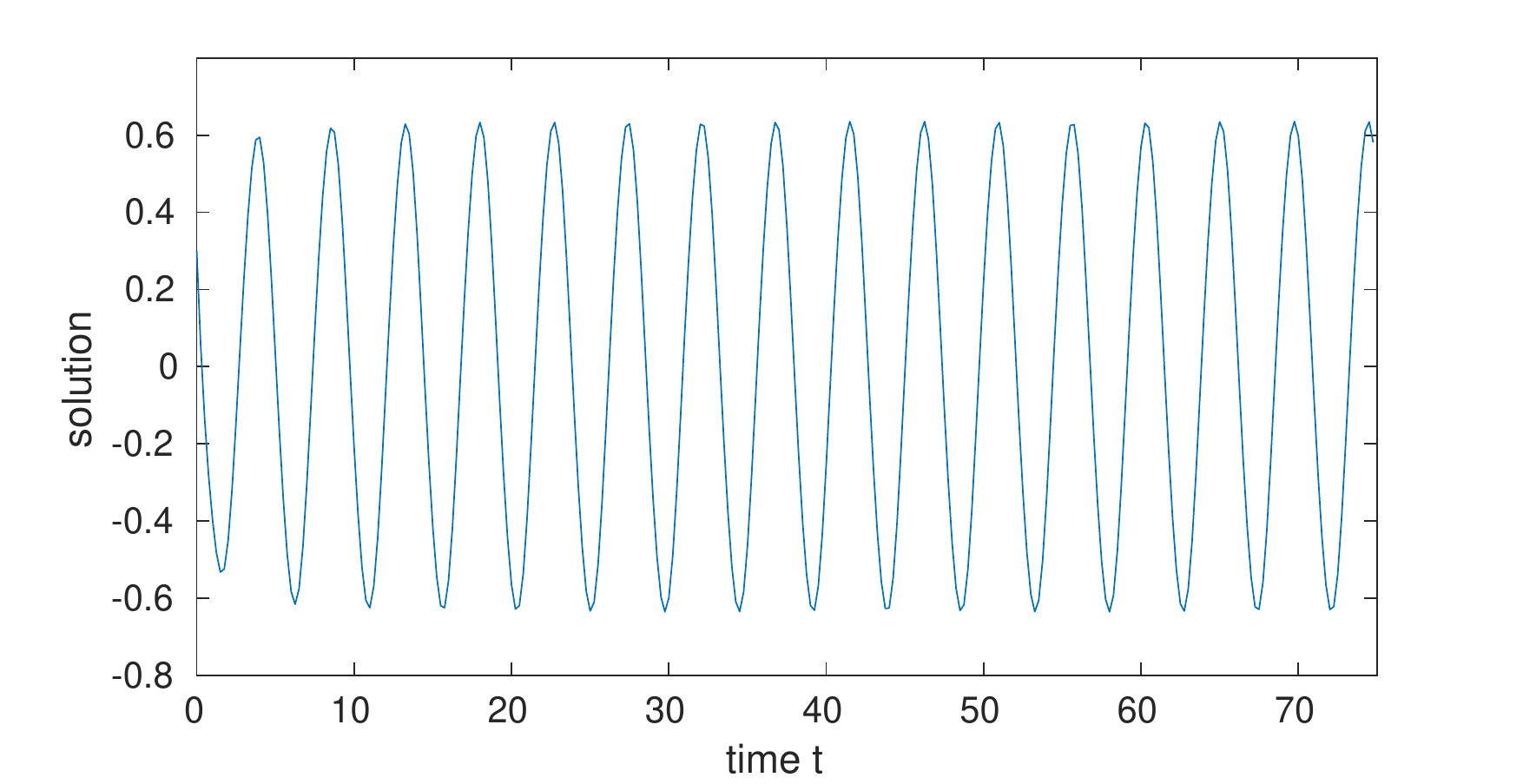}
\includegraphics[scale=.43]{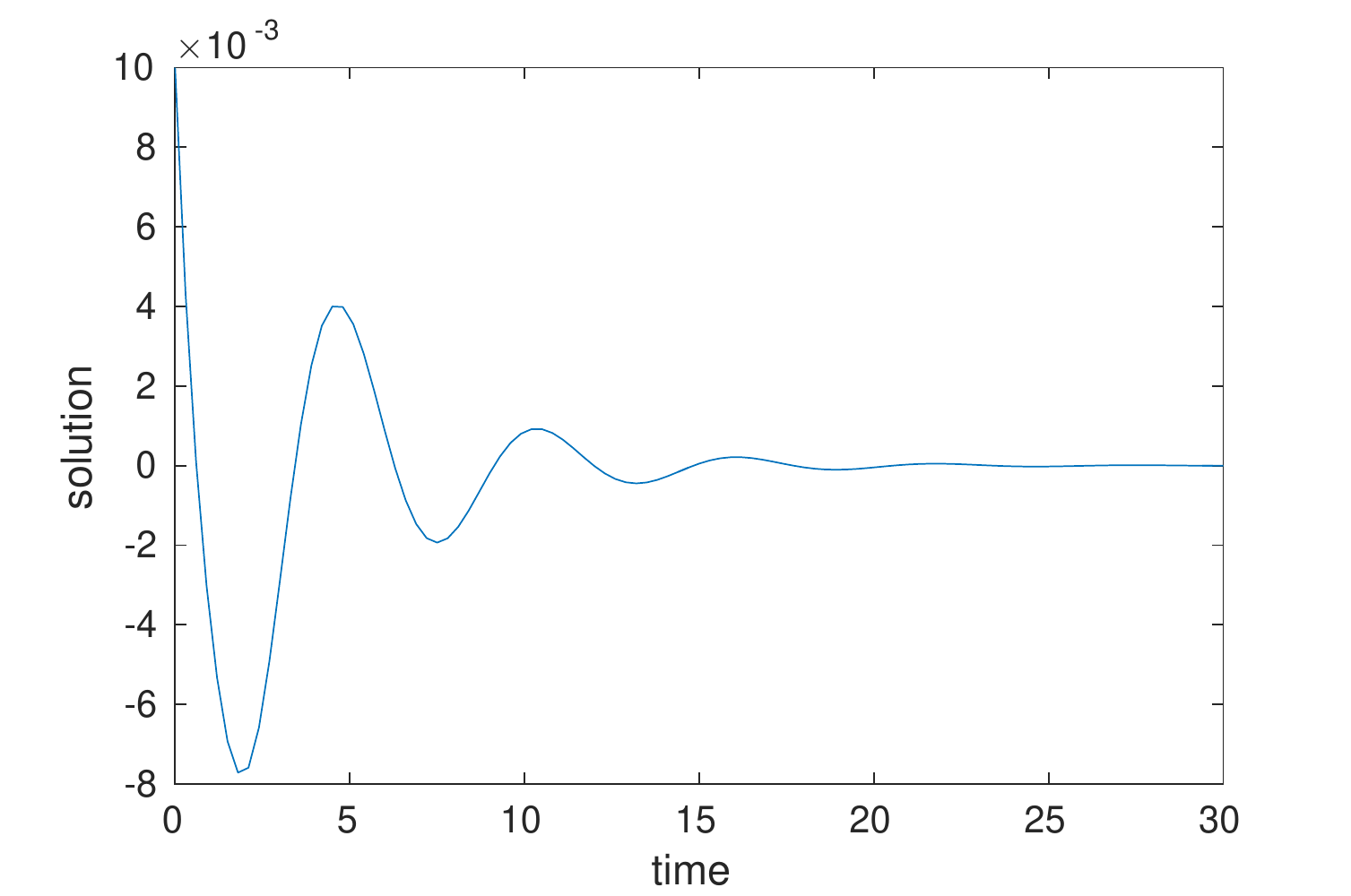}
\caption{Time evolution of the system at a given position in space when $\hat c=-4$ (left) and when $\hat c=-0.5$ (right).}
\label{fig:time_simulation_around_Hopf}
\end{center}
\end{figure}
\begin{figure}
\begin{center}
\includegraphics[scale=.55]{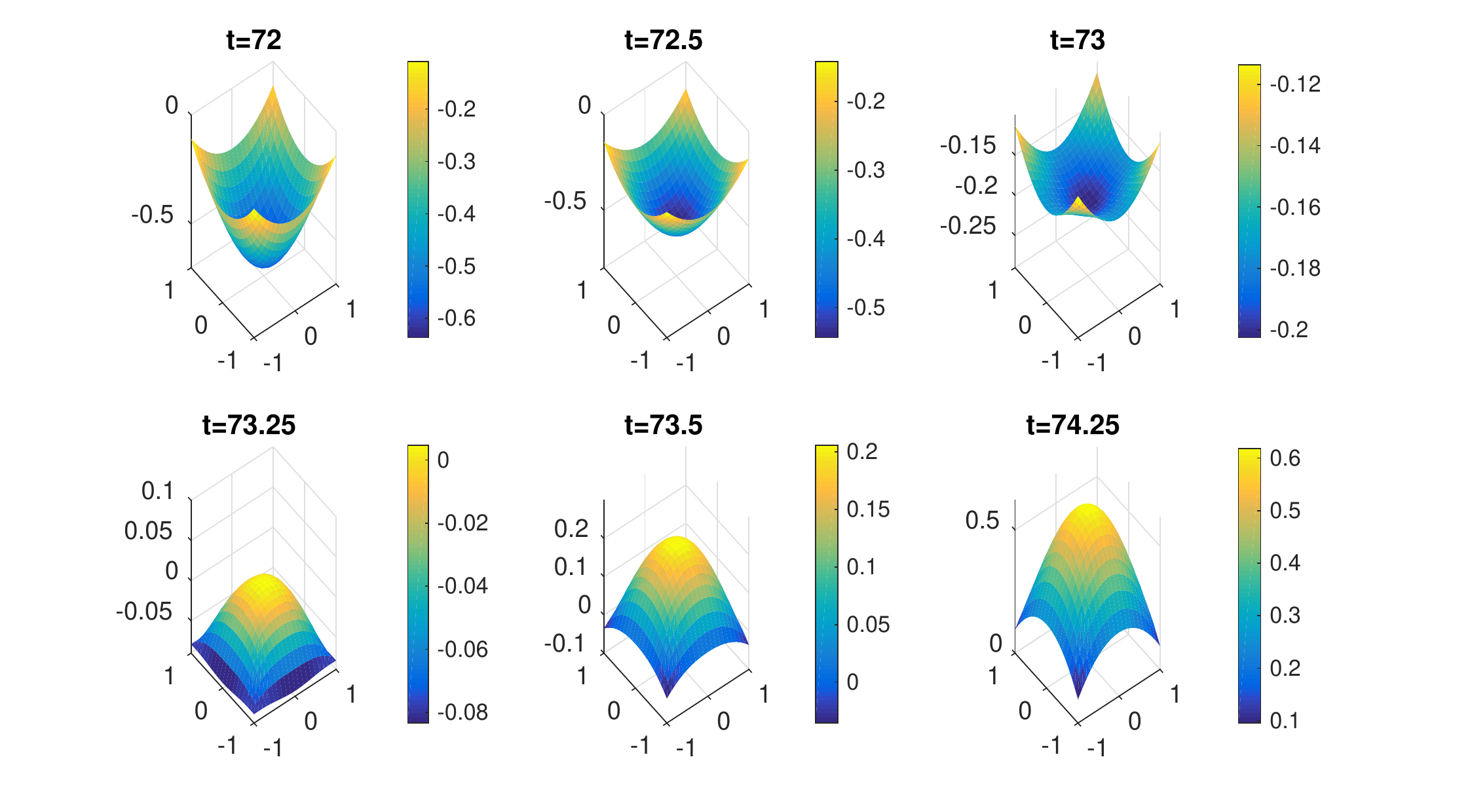}
\caption{Time evolution of the system during half of a time period when $\hat c=-4$.}\label{fig:time_simulation_beyond_Hopf}
\end{center}
\end{figure}
\section{Conclusions}
\label{sec:conclusions}
We studied a neural field model with transmission delays and a connectivity kernel that is a linear combination of exponentials. Motivated by applications in neuroscience, we used a planar spatial domain, in particular a rectangle. This however, made the analysis more challenging. 

We investigated in detail a model with a connectivity kernel that is a single exponential.  To study the dynamics of the linearized equation, we completely characterised the spectrum. We constructed eigenvectors as solutions of the characteristic integral equation. We employed the fact that the integral equation is equivalent to a partial differential equation with a Robin-type boundary condition. This PDE can be separated into two differential equations of Sturm-Liouville type. We constructed a  basis of solutions to these differential equations that is complete in $L^2$. These basis functions allowed us to determine whether $z\in \mathbb{C}$ is part of the spectrum. 

It is quite rare that one can completely characterise the spectrum of an operator acting on multivariate  functions, such as partial differential or integral operators. It is sometimes possible to find some eigenvalues, but here we proved that we found them all. We investigated a numerical example of a population of inhibitory neurons. Using the developed theory, we detected a supercritical Hopf bifurcation. 

The classical example of an excitatory and inhibitory population of neurons can be modeled using a connectivity kernel of two exponentials. In the special case when the rectangle is a square, we found eigenvalues and eigenfunctions, but we cannot conjecture that there are no more. 

\section*{Acknowledgments}
M.~Polner was supported by the J\'anos Bolyai Research Scholarship of the Hungarian Academy of Sciences, the Hungarian Scientific Research Fund, Grant No. K129322 and SNN125119. Her research was also supported by grant TUDFO/47138-1/2019-ITM of the
 Ministry for Innovation and Technology, Hungary.

M. Polner would like to thank Prof. L\'aszl\'o Stach\'o for their valuable discussions on Section~\ref{s:SLP}.

\bibliographystyle{abbrv} 
\bibliography{referencesPM}

\end{document}